\documentclass[12pt]{article}
\usepackage[top=30truemm,bottom=30truemm,left=20truemm,right=20truemm]{geometry}
\usepackage{amsmath}
\usepackage{amsthm}
\usepackage{amssymb}
\usepackage{amsfonts}
\usepackage{mathrsfs}
\usepackage{mathtools}
\usepackage{wrapfig}
\usepackage{graphicx}
\usepackage{cite}
\usepackage{comment}
\usepackage{autobreak}
\usepackage{enumitem}

\numberwithin{equation}{section}

\theoremstyle{plain}
\newtheorem{thm}{\textbf{Theorem}}[section]
\newtheorem*{thm*}{\textbf{Theorem}}
\newtheorem{prop}[thm]{\textbf{Proposition}}
\newtheorem*{prop*}{\textbf{Proposition}}
\newtheorem{lem}[thm]{\textbf{Lemma}}
\newtheorem*{lem*}{\textbf{Lemma}}

\theoremstyle{definition}

\newtheorem*{dfn*}{\textbf{Definition}}
\newtheorem{rem}[thm]{\textbf{Remark}}
\newtheorem*{rem*}{\textbf{Remark}}

\newtheorem*{ass*}{\textbf{Assumption}}

\mathtoolsset{showonlyrefs}

\setlist{align=left}

\title{Classification of unstable travelling wave solutions to KdV type equations}
\date{}
\author{Kaito KOKUBU%
    \footnote{
    Department of Mathematics, Graduate School of Science, Tokyo University of Science.\\
    1-3 Kagurazaka, Shinjuku-ku, Tokyo 162-8601, Japan.\\
    e-mail: \texttt{1123701@ed.tus.ac.jp}}
}

\begin{document}

\maketitle

\renewcommand{\thefootnote}{\fnsymbol{footnote}}
\footnote[0]{\textbf{2020 Mathematics Subject Classification.} 35Q53, 35R11, 35B35, 35C07.}
\footnote[0]{\textbf{Keywords and Phrases.} Korteweg--de Vries equation; Gardner equation; Benjamin--Ono equation; Fractional Laplacian; Travelling wave; instability; Ground state.}
\renewcommand{\thefootnote}{\arabic{footnote}}

\begin{abstract}
    We study travelling wave solutions to Korteweg--de Vries type equations which have double power nonlinearities with integer indices, such as the Gardner equation, and fractional dispersion.
    Whether these equations have ground state solutions depends on signatures of nonlinearities and parity combinations of the two indices.
    The aim of this study is to give the classification of phenomena of travelling wave solutions from the perspective of the signatures and parities of the indices.
    In this paper, we focus on unstable travelling wave solutions.
\end{abstract}


\section{Introduction}

In this paper, we consider the instability of travelling wave solutions to Korteweg--de Vries type equations
\begin{equation}
    \partial_{t}u + \partial_{x}f(u) - \partial_{x} D_{x}^{\sigma} u = 0, \quad (t, x) \in \mathbb{R} \times \mathbb{R}, \label{eq:Intro_gKdV_general;gkdv_inst}
\end{equation}
where $u = u(t,x)$ is a real-valued unknown function, $\sigma$ is a real number satisfying $1 \leq \sigma \leq 2$, and the operator $D_{x}^{\sigma}$ is the Fourier multiplier with symbol $|\xi|^{\sigma}$.
This operator is also denoted by $(-\partial_{x}^{2})^{\sigma/2}$.

When $\sigma = 2$ and $f(s) = s^{2} (s \in \mathbb{R})$, equation \eqref{eq:Intro_gKdV_general;gkdv_inst} coincides with the Korteweg--de Vries equation, which describes physical dynamics of waves in shallow water (see e.g.\ \cite{Korteweg-deVries}).
When $\sigma = 1$ and $f(s) = s^{2}$, equation \eqref{eq:Intro_gKdV_general;gkdv_inst} is the Benjamin--Ono equation, which physically describes dynamics of internal waves in stratified fluids (see e.g.\ \cite{Benjamin,Ono_1975}).

A travelling wave solution is a solution to \eqref{eq:Intro_gKdV_general;gkdv_inst} of the form $u(t, x) = \phi(x - ct)$, where $c$ is a positive constant representing the speed of the wave, and $\phi \in H^{\sigma/2}(\mathbb{R})$ is a nontrivial solution to the stationary problem
\begin{equation}
    D_{x}^{\sigma} \phi + c \phi - f(\phi) = 0, \quad x \in \mathbb{R}. \label{eq:Intro_SP_general;gkdv_inst}
\end{equation}
We say that $\phi \in H^{\sigma/2}(\mathbb{R})$ is a ground state solution to \eqref{eq:Intro_SP_general;gkdv_inst} if $\phi$ satisfies
\begin{equation}
    S_{c}(\phi) = \inf\{ S_{c}(v): \ v \in H^{\sigma/2}(\mathbb{R}) \setminus \{0\}, \ \text{$v$ is a solution to \eqref{eq:Intro_SP_general;gkdv_inst}} \},
\end{equation}
where the functional $S_{c}$ is the action functional corresponding to equation \eqref{eq:Intro_SP_general;gkdv_inst} defined in $H^{\sigma/2}(\mathbb{R})$ as $S_{c}(v) \coloneq E(v) + cM(v)$ with the energy functional $E$ and the mass functional $M$ defined as
\begin{equation}
    E(v) \coloneq \frac{1}{2} \| D_{x}^{\sigma/2}v \|_{L^{2}}^{2} - \int_{\mathbb{R}} F(v) \, dx, \quad M(v) = \frac{1}{2} \| v \|_{L^{2}}^{2},
\end{equation}
where $F(s)$ is the primitive function of the nonlinearity $f(s)$.
In this paper, we mainly consider power type nonlinearities, which is sufficient for the functional $S_{c}$ to belong to $C^{2}(H^{\sigma/2}(\mathbb{R}), \mathbb{R})$.
Moreover, we remark that $v \in H^{\sigma/2}(\mathbb{R})$ is a solution to \eqref{eq:Intro_SP_general;gkdv_inst} if and only if $S_{c}'(v) = 0$, where $G'$ denotes the Fr\'{e}chet derivative of a functional $G$.
In this paper, we mainly consider travelling wave solutions to \eqref{eq:Intro_gKdV_general;gkdv_inst} constructed with ground state solutions to \eqref{eq:Intro_SP_general;gkdv_inst}.

Next, we define the stability and instability of travelling wave solutions.

For $r > 0$ and $v \in H^{\sigma/2}(\mathbb{R})$, we put
\begin{equation}
    U_{r}(v) \coloneq \{ u \in H^{\sigma/2}(\mathbb{R}): \ \inf_{y \in \mathbb{R}} \| u - v(\cdot - y) \|_{H^{\sigma/2}}^{2} < r \}.
\end{equation}
Let $\phi \in H^{\sigma/2}(\mathbb{R})$ be a nontrivial solution to \eqref{eq:Intro_SP_general;gkdv_inst}.
We say that a travelling wave solution $\phi(x-ct)$ to \eqref{eq:Intro_gKdV_general;gkdv_inst} is stable if, for any $\varepsilon > 0$, there exist $\delta > 0$ such that if $u_{0} \in U_{\delta}(\phi)$, then the time-global $u(t) \in C([0, \infty), H^{\sigma/2}(\mathbb{R}))$ to \eqref{eq:Intro_gKdV_general;gkdv_inst} exists and satisfies that $u(t) \in U_{\varepsilon}(\phi)$ holds for all $t \geq 0$.
Otherwise, we say that a travelling wave solution $\phi(x-ct)$ to \eqref{eq:Intro_gKdV_general;gkdv_inst} is unstable.

The stability and instability of travelling wave solutions to \eqref{eq:Intro_gKdV_general;gkdv_inst} with single power nonlinearities $f(s) = s^{p}$ ($p \in \mathbb{N}$, $2 \leq p < \infty$) have been studied well.
In this case, it is known that the stationary problem \eqref{eq:Intro_SP_general;gkdv_inst} has the unique positive and even ground state solution (existence: Weinstein~\cite{Weinstein}, uniqueness for $1 \leq \sigma < 2$: Frank--Lenzmann~\cite{Frank-Lenzmann}).
By an abstract theory developed by Bona--Souganidis--Strauss~\cite{BSS1987}, we can find that the travelling wave solution to \eqref{eq:Intro_gKdV_general;gkdv_inst} constructed with the positive ground state solution to \eqref{eq:Intro_SP_general;gkdv_inst} is stable if $p < 2\sigma + 1$, or unstable if $p > 2\sigma + 1$
(for the stability results, see also \cite{Weinstein}).
According to \cite{BSS1987}, travelling wave solutions are stable if $\partial_{c}^{2} d(c) > 0$, or unstable if $\partial_{c}^{2} d(c) < 0$, where $d(c) = S_{c}(\psi_{c})$ with $\psi_{c}$ denoting the ground state solution for $c > 0$.
When nonlinearities of \eqref{eq:Intro_gKdV_general;gkdv_inst} are single power ones, we can find the scaling property of the ground state solution to \eqref{eq:Intro_SP_general;gkdv_inst} such that
\begin{equation}
    \psi_{c}(x) = c^{1/(p-1)} \psi_{1}(c^{1/\sigma} x). \label{eq:Intro_0010;gkdv_inst}
\end{equation}
This property allows us to calculate $\partial_{c}^{2} d(c)$ so easily that we can see that $\partial_{c}^{2} d(c) > 0$ holds if $p < 2\sigma + 1$, or $\partial_{c}^{2} d(c) < 0$ if $p > 2\sigma + 1$.
When $\sigma = 2$ and $p = 2 \sigma + 1 = 5$, where the method of \cite{BSS1987} is not applicable, Martel--Merle~\cite{Martel-Merle_2001} proved that the travelling wave solution to \eqref{eq:Intro_gKdV_general;gkdv_inst} is unstable by using a virial type functional (see also \cite{Farah-Holmer-Roudenko}).
Here we remark that the exponent $2 \sigma + 1$ is the $L^{2}$-critical one for the KdV type equation \eqref{eq:Intro_gKdV_general;gkdv_inst}.

The purpose of this study is to analyze the stability and instability of travelling wave solutions to \eqref{eq:Intro_gKdV_general;gkdv_inst} with integer-indices double power nonlinearities such as
\begin{equation}
    \partial_{t}u + \partial_{x} (au^{p} + u^{q}) - \partial_{x} D_{x}^{\sigma} u = 0, \quad (t,x) \in \mathbb{R} \times \mathbb{R}, \label{eq:Intro_gKdV_DoublePower;gkdv_inst}
\end{equation}
where $a \in \{ \pm 1 \}$ and $p, q \in \mathbb{N}$ satisfying $2 \leq p < q < \infty$.
The remarkable example of these equations is the case that $\sigma = 2$, $p = 2$, and $q = 3$, where the equation \eqref{eq:Intro_gKdV_DoublePower;gkdv_inst} coincides with the so-called Gardner equation, which was introduced by Miura--Gardner--Kruskal~\cite{Miura-Gardner-Kruskal_II}.

Throughout this paper, we assume the local well-posedness of the Cauchy problem associated with \eqref{eq:Intro_gKdV_DoublePower;gkdv_inst} in the energy space $H^{\sigma/2}(\mathbb{R})$ as follows:
\begin{ass*}
    Let $1 \leq \sigma \leq 2$.
    Then, for any $u_{0} \in H^{\sigma/2}(\mathbb{R})$, there exist $T = T(\| u_{0} \|_{H^{\sigma/2}}) > 0$ and a unique solution $u(t) \in C([0, T), H^{\sigma/2}(\mathbb{R}))$ to \eqref{eq:Intro_gKdV_DoublePower;gkdv_inst} with $u(0) = u_{0}$ which satisfy the following conservation laws:
    \begin{equation}
        E(u(t)) = E(u_{0}), \quad M(u(t)) = M(u_{0}), \quad t \in [0, T).
    \end{equation}
\end{ass*}
We remark that Molinet--Tanaka~\cite{Molinet-Tanaka_2024} showed that this assumption actually holds if $4/3 \leq \sigma \leq 2$.

The stationary problem derived from \eqref{eq:Intro_gKdV_DoublePower;gkdv_inst} is the following:
\begin{equation}
    D_{x}^{\sigma} \phi + c \phi - a \phi^{p} - \phi^{q} = 0, \quad x \in \mathbb{R}. \label{eq:Intro_dSPc;gkdv_inst}
\end{equation}
The author has studied existence of ground state solutions to \eqref{eq:Intro_dSPc;gkdv_inst} and found that parity combinations of indices $p$ and $q$, and the signature $a$ affect the existence phenomena of ground state solutions to \eqref{eq:Intro_dSPc;gkdv_inst} as follows:
\begin{thm}[\cite{Kokubu2024,kokubu2025stability}] \label{thm:Intro_dSPc_previous;gkdv_inst}
    Let $1 \leq \sigma \leq 2$, $p, q \in \mathbb{N}$, $2 \leq p < q < \infty$, and $c > 0$.
    Then the following properties hold:
    
    \begin{enumerate}[label=\textup{(\Roman*)} \ ]
        \item \underline{The case $a = +1$}

        If $q$ is odd, then there exists a positive ground state solution to \eqref{eq:Intro_dSPc;gkdv_inst}.

        \item \underline{The case $a = -1$}

        \begin{enumerate}[label=\textup{(\roman*)} \ ]
            \item If $p$ is odd, then there exists a positive ground state solution to \eqref{eq:Intro_dSPc;gkdv_inst}.
            \item If $p$ is even and $q$ is odd, then there exists a negative ground state solution to \eqref{eq:Intro_dSPc;gkdv_inst}.
        \end{enumerate}
    \end{enumerate}

    Moreover, a ground state solution $\phi$ to \eqref{eq:Intro_dSPc;gkdv_inst} obtained above is characterized as follows:
    \begin{equation}
        S_{c}(\phi) = \inf\{ S_{c}(v): \ v \in H^{\sigma/2}(\mathbb{R}) \setminus \{0\}, \ K_{c}(v) = 0 \}, \label{eq:Intro_GroundState_Characterization;gkdv_inst}
    \end{equation}
    where $K_{c}$ is the Nehari functional defined as $K_{c}(v) \coloneq \langle S_{c}'(v), v \rangle$.
\end{thm}

\begin{rem}
    \begin{enumerate}
        \item Each ground state solution obtained in Theorem \ref{thm:Intro_dSPc_previous;gkdv_inst} can be taken as an even function, and be decreasing in $|x|$ if positive, or increasing in $|x|$ if negative.
        Hereafter, whenever we take a ground state solution to \eqref{eq:Intro_dSPc;gkdv_inst}, we always consider an even one.

        \item We can see that any solution $v \in H^{\sigma/2}(\mathbb{R})$ to \eqref{eq:Intro_dSPc;gkdv_inst} belongs to $H^{\infty}(\mathbb{R}) = \bigcap_{s \geq 0} H^{s}(\mathbb{R})$.
        For a proof of this, see e.g.\ \cite[Lemma 2.6]{Kokubu2024}.

        \item In Theorem \ref{thm:Intro_dSPc_previous;gkdv_inst}, we only consider cases that we can find a ground state solution, while there are other cases where we can find a nontrivial solution which is (possibly) not a ground state solution.
        For instance, under condition (II-2) in Theorem \ref{thm:Intro_dSPc_previous;gkdv_inst}, we can find a positive solution $v \in H^{\sigma/2}(\mathbb{R})$ to \eqref{eq:Intro_dSPc;gkdv_inst} which is not a ground state solution.
        Namely, it holds that $S_{c}(\phi_{c}) < S_{c}(v)$, where $\phi_{c}$ is a negative ground state solution to \eqref{eq:Intro_dSPc;gkdv_inst}.
        (More precisely, in this case, any positive solution $v \in H^{\sigma/2}(\mathbb{R})$ to \eqref{eq:Intro_dSPc;gkdv_inst} does not become a ground state solution.)
        For details, see \cite[Theorem 2.5]{Kokubu2024}.
    \end{enumerate}
\end{rem}

When it comes to other nonlinear dispersive equations such as nonlinear Schr\"{o}dinger equations, we usually consider power type nonlinearities of form $|u|^{p-1}u$ ($p > 1$). 
These nonlinearities have a symmetricity that if $\phi$ solves the corresponding stationary problem, then $-\phi$ also solves it.
However, similar symmetricity does not always hold for KdV type equations, whose nonlinearities are of form $u^{p}$ ($p \in \mathbb{N}$, $p \geq 2$).
Therefore, it can be said that the dependency of existence of the ground state solutions on the signatures and the parities of indices, like Theorem \ref{thm:Intro_dSPc_previous;gkdv_inst}, is a specific phenomenon to KdV type equation.
Due to Theorem \ref{thm:Intro_dSPc_previous;gkdv_inst}, the stability and instability of travelling wave solutions to \eqref{eq:Intro_gKdV_DoublePower;gkdv_inst} can vary depending on the signatures and the parities of indices, which is also seemed to be specific to KdV type equations.
The goal of this study is to give the classification of the stability and instability phenomena from the perspective of the parities and the signature.

In the author's previous paper~\cite{kokubu2025stability}, we classified conditions where travelling wave solutions to \eqref{eq:Intro_gKdV_DoublePower;gkdv_inst} constructed by ground state solutions to \eqref{eq:Intro_dSPc;gkdv_inst} obtained in Theorem \ref{thm:Intro_dSPc_previous;gkdv_inst} are stable, whose discussion was based on the method by Grillakis--Shatah--Strauss~\cite{GSS87}.
This paper is a sequel to \cite{kokubu2025stability} and deals with the instability of travelling wave solutions to \eqref{eq:Intro_gKdV_DoublePower;gkdv_inst}.
Now we state the main results of this paper: the classification of conditions for unstable travelling wave solutions.

\begin{thm} \label{thm:Intro_Main_Instability;gkdv_inst}
    Let $1 \leq \sigma \leq 2$, $p, q \in \mathbb{N}$, and $2 \leq p < q < \infty$.
    Then the following properties hold:

    \begin{enumerate}[label=\textup{(\Roman*) \ }]
        \item \underline{The case $a = +1$}

        Assume that $q$ is odd and let $\phi_{c}$ be a positive ground state solution to \eqref{eq:Intro_dSPc;gkdv_inst} for $c > 0$.

        \begin{enumerate}[label=\textup{(\arabic*) \ }]
            \item If $2\sigma + 1 \leq p < q$, then a travelling wave solution $\phi_{c}(x - ct)$ to \eqref{eq:Intro_gKdV_DoublePower;gkdv_inst} is unstable for all $c > 0$.

            \item If $p < 2\sigma + 1 < q$, then there exists $c_{1} \in (0, \infty)$ such that a travelling wave solution $\phi_{c}(x - ct)$ to \eqref{eq:Intro_gKdV_DoublePower;gkdv_inst} is unstable for $c \in (c_{1}, \infty)$.
        \end{enumerate}

        \item \underline{The case $a = -1$}

        \begin{enumerate}[label=\textup{(\arabic*)} \ ]
            \item Assume that $p$ is odd, and let $\phi_{c}$ be a positive ground state solution to \eqref{eq:Intro_dSPc;gkdv_inst} for $c > 0$.

            \begin{enumerate}[label=\textup{(\roman*)} \ ]
                \item If $p < q = 2\sigma + 1$, then a travelling wave solution $\phi_{c}(x - ct)$ to \eqref{eq:Intro_gKdV_DoublePower;gkdv_inst} is unstable for all $c > 0$.

                \item If $p \leq 2\sigma + 1 < q$, then a travelling wave solution $\phi_{c}(x - ct)$ to \eqref{eq:Intro_gKdV_DoublePower;gkdv_inst} is unstable for all $c > 0$.

                \item If $2 \sigma + 1 < p < q$, then there exists $c_{2} \in (0, \infty)$ such that a travelling wave solution $\phi_{c}(x - ct)$ to \eqref{eq:Intro_gKdV_DoublePower;gkdv_inst} is unstable for $c \in (c_{2}, \infty)$.
            \end{enumerate}

            \item Assume that $p$ is even and $q$ is odd, and let $\phi_{c}$ be a negative ground state solution to \eqref{eq:Intro_dSPc;gkdv_inst}.

            \begin{enumerate}[label=\textup{(\roman*)} \ ]
                \item If $2\sigma + 1 \leq p < q$, then a travelling wave solution $\phi_{c}(x - ct)$ to \eqref{eq:Intro_gKdV_DoublePower;gkdv_inst} is unstable for all $c > 0$.

                \item If $p < 2\sigma + 1 < q$, then there exists $c_{3} \in (0, \infty)$ such that a travelling wave solution $\phi_{c}(x - ct)$ to \eqref{eq:Intro_gKdV_DoublePower;gkdv_inst} is unstable for $c \in (c_{3}, \infty)$.
            \end{enumerate}
        \end{enumerate}
    \end{enumerate}
\end{thm}

\begin{rem}
    We can find only a few examples which satisfy conditions (II-1-i) or (II-2-i) in Theorem \ref{thm:Intro_Main_Instability;gkdv_inst}.
    For instance, $(\sigma, p, q) = (2, 3, 5), (3/2, 3, 4)$ are the only cases satisfying condition (II-1-i).
    For condition (II-2-i), we can find only one example that $(\sigma, p, q) = (3/2, 4, 5)$.
\end{rem}

The method of \cite{BSS1987} is hardly applicable to nonlinear dispersive equations with generalized nonlinearities, including double power ones, due to lack of scaling properties of ground state solutions like \eqref{eq:Intro_0010;gkdv_inst}, which makes it difficult to compute $\partial_{c}^{2} d(c)$.
Therefore, we need to focus on individual equations to analyze the stability and instability of travelling wave solutions to \eqref{eq:Intro_gKdV_general;gkdv_inst}.

The organization of this paper is as follows.
In Section \ref{section:Instability_Sufficient;gkdv_inst}, we introduce a sufficient condition for travelling wave solutions to \eqref{eq:Intro_gKdV_DoublePower;gkdv_inst} to be unstable (Proposition \ref{prop:Instability_Sufficient;gkdv_inst}).
The proof of Proposition \ref{prop:Instability_Sufficient;gkdv_inst} is based on the variational characterization \eqref{eq:Intro_GroundState_Characterization;gkdv_inst} of ground state solutions to \eqref{eq:Intro_dSPc;gkdv_inst}.
In Section \ref{section:Proof_of_Theorem;gkdv_inst}, we observe conditions which satisfy the sufficient condition given in Proposition \ref{prop:Instability_Sufficient;gkdv_inst} to conclude Theorem \ref{thm:Intro_Main_Instability;gkdv_inst}.

\subsection*{Notations}
\begin{itemize}
    \item For a function $u$, both $\mathscr{F} u$ and $\hat{u}$ denote the Fourier transform of $u$ defined as
    \begin{equation}
        \mathscr{F}u(\xi) = \hat{u}(\xi) \coloneq \frac{1}{\sqrt{2\pi}} \int_{\mathbb{R}} e^{-ix\xi} u(x) \, dx.
    \end{equation}

    \item For $s > 0$ and $u, v \in H^{s}(\mathbb{R})$, we define an inner product in $H^{s}(\mathbb{R})$ as
    \begin{equation}
        (u, v)_{H^{s}} \coloneq (D_{x}^{s}u, D_{x}^{s}v)_{L^{2}} + (u, v)_{L^{2}}.
    \end{equation}
    Let $\| \cdot \|_{H^{\sigma/2}}$ denote the norm induced by this inner product.

    \item For $x \in \mathbb{R}$, we put $\langle x \rangle \coloneq (1 + x^{2})^{1/2}$.

    \item We let $\langle f, u \rangle$ denote a dual product between $f \in H^{-\sigma/2}(\mathbb{R})$ and $u \in H^{\sigma/2}(\mathbb{R})$ if there is no confusion.
    In section \ref{section:Instability_Sufficient;gkdv_inst}, when we should clarify the sense of dual products, we write $\langle f, u \rangle_{X^{\ast}, X}$ for $f \in X^{\ast}$ and $u \in X$, where $X$ is a Banach space.

    \item In estimates, we often use the same letter to denote positive constants, whose values may change from line to line.
\end{itemize}

\section{Sufficient condition for unstable travelling wave solutions} \label{section:Instability_Sufficient;gkdv_inst}

In this section, we always assume one of conditions (I), (II-1), or (II-2) in Theorem \ref{thm:Intro_dSPc_previous;gkdv_inst}, and let $\phi_{c}$ be a ground state solution to \eqref{eq:Intro_dSPc;gkdv_inst} for $c > 0$ obtained in Theorem \ref{thm:Intro_dSPc_previous;gkdv_inst}.

We state the sufficient condition for travelling wave solutions to be unstable.
\begin{prop} \label{prop:Instability_Sufficient;gkdv_inst}
    For $\lambda > 0$, we put
    \begin{equation}
        \phi_{c}^{\lambda}(x) \coloneq \lambda^{1/2} \phi_{c}(\lambda x). \label{eq:Instability_0010;gkdv_inst}
    \end{equation}
    If $\partial_{\lambda}^{2}S_{c}(\phi_{c}^{\lambda}) |_{\lambda=1} < 0$, then a travelling wave solution $\phi_{c}(x - ct)$ to \eqref{eq:Intro_gKdV_DoublePower;gkdv_inst} is unstable.
\end{prop}
We remark that the condition $\partial_{\lambda}^{2}S_{c}(\phi_{c}^{\lambda}) |_{\lambda=1} < 0$ was introduced by Ohta~\cite{Ohta_1995_DSinst}.

In the following, we prove Proposition \ref{prop:Instability_Sufficient;gkdv_inst}.
The proof is based on the discussion in \cite{Riano-Roudenko_2022} with some modifications.

First we observe some properties of ground state solutions to \eqref{eq:Intro_dSPc;gkdv_inst}.
Here we state some lemmas which hold for the ground state solutions.

\begin{lem} \label{lem:Variational_PositivenessofCoupling;gkdv_inst}
    If $v \in H^{\sigma/2}(\mathbb{R})$ satisfies $\langle K_{c}'(\phi_{c}), v \rangle = 0$, then it holds that
    \begin{equation}
        \langle S_{c}''(\phi_{c})v, v \rangle \geq 0.
    \end{equation}
\end{lem}
\begin{proof}
    See \cite[Lemma 4]{Ohta_2014}.
\end{proof}

\begin{lem} \label{lem:Instability_UnstableDir_1;gkdv_inst}
    It holds that $\langle S_{c}''(\phi_{c}) \phi_{c}, \phi_{c} \rangle < 0$.
\end{lem}
\begin{proof}
    Considering the graph of the function $(0, \infty) \ni  \lambda \mapsto S_{c}(\lambda \phi_{c}) \in \mathbb{R}$, we can see that
    \begin{equation}
        \partial_{\lambda} S_{c}(\lambda \phi_{c}) |_{\lambda=1} = \langle S_{c}'(\phi_{c}), \phi_{c} \rangle = 0.
    \end{equation}
    By the convexity of this function, we obtain
    \begin{equation}
        \langle S_{c}''(\phi_{c}) \phi_{c}, \phi_{c} \rangle = \partial_{\lambda}^{2} S_{c}(\lambda \phi_{c}) |_{\lambda=1} < 0.
    \end{equation}
\end{proof}

Next we observe the decay estimate of ground state solutions to \eqref{eq:Intro_dSPc;gkdv_inst}.

When $\sigma = 2$, it is well known that a ground state solution $\phi_{c}$ to \eqref{eq:Intro_dSPc;gkdv_inst} decays exponentially, that is,
\begin{equation}
    |\phi_{c}(x)| + |\partial_{x} \phi_{c}(x)| \leq C e^{-\delta |x|} \label{eq:Instability_Preliminaries_0010;gkdv_inst}
\end{equation}
for $x \in \mathbb{R}$ with some constants $C, \delta > 0$.

If $1 \leq \sigma < 2$, we can see that a ground state solution $\phi_{c}$ to \eqref{eq:Intro_dSPc;gkdv_inst} decays as polynomially as follows;
for any $l \in \mathbb{Z}_{+}$, it holds that
\begin{equation}
    |x^{l} \partial_{x}^{l} \phi_{c}(x)| \leq C \langle x \rangle^{-(1 + \sigma)} \label{eq:Instability_Preliminaries_0020;gkdv_inst}
\end{equation}
for $x \in \mathbb{R}$ with some constant $C > 0$.
The similar decay estimate for the ground state solutions to stationary problems with single power nonlinearity has been proved by Ria\~{n}o--Roudenko~\cite{Riano-Roudenko_2022}.
Following their discussion, we can prove the decay estimate \eqref{eq:Instability_Preliminaries_0020;gkdv_inst}.
For readers' convenience, we will give the proof of \eqref{eq:Instability_Preliminaries_0020;gkdv_inst} in Section \ref{section:DecayEst_of_GS;gkdv_inst}.

Now we put
\begin{equation}
    \Lambda \phi_{c} \coloneq \partial_{\lambda} \phi_{c}^{\lambda} |_{\lambda=1} = \frac{1}{2}\phi_{c} + x \partial_{x}\phi_{c}. \label{eq:Instability_UnstableDir_0010;gkdv_inst}
\end{equation}
By \eqref{eq:Instability_Preliminaries_0010;gkdv_inst} or \eqref{eq:Instability_Preliminaries_0020;gkdv_inst}, we can see that $\Lambda \phi_{c} \in H^{1 + \sigma/2}(\mathbb{R}) \cap L^{1}(\mathbb{R})$.

\begin{lem}
    It holds that
    \begin{align}
        & (\phi_{c}, \Lambda \phi_{c})_{L^{2}} = 0, \label{eq:Instability_UnstableDir_0020;gkdv_inst} \\
        & \langle S_{c}''(\phi_{c}) \Lambda \phi_{c}, \Lambda \phi_{c} \rangle < 0. \label{eq:Instability_UnstableDir_0030;gkdv_inst}
    \end{align}
\end{lem}
\begin{proof}
    Since the scaling \eqref{eq:Instability_0010;gkdv_inst} is $L^{2}$-norm invariant, we have
    \begin{equation}
        0 = \partial_{\lambda} M(\phi_{c}^{\lambda}) |_{\lambda=1} = \langle M'(\phi_{c}), \Lambda \phi_{c} \rangle = (\phi_{c}, \Lambda \phi_{c})_{L^{2}}.
    \end{equation}
    
    Moreover, since $S_{c}'(\phi_{c}) = 0$ in $H^{-\sigma/2}(\mathbb{R})$, we obtain
    \begin{equation}
        \langle S_{c}''(\phi_{c}) \Lambda \phi_{c}, \Lambda \phi_{c} \rangle = \partial_{\lambda}^{2} S_{c}(\phi_{c}^{\lambda}) |_{\lambda = 1} < 0.
    \end{equation}
\end{proof}

Here we recall the modulation theory around a ground state solution $\phi_{c}$ to \eqref{eq:Intro_dSPc;gkdv_inst}.

\begin{lem} \label{lem:Instability_Modulation_Modulation;gkdv_inst}
    There exist $\varepsilon_{1} > 0$ and a unique $C^{1}$-mapping $\tilde{z}\colon U_{\varepsilon_{1}}(\phi_{c}) \rightarrow \mathbb{R}$ which satisfy $\tilde{z}(\phi_{c}) = 0$ and the following properties for all $u \in U_{\varepsilon_{1}}(\phi_{c})$ and all $y \in \mathbb{R}$:
    
    \begin{enumerate}[label=\textup{(\roman*)} \ ]
        \item $\left( u(\cdot + \tilde{z}(u)), \partial_{x} \phi_{c} \right)_{L^{2}} = 0$,

        \item $\tilde{z}\left( u (\cdot + y) \right) = \tilde{z}(u) - y$,

        \item $\displaystyle \tilde{z}'(u) = \frac{\partial_{x} \phi_{c}(\cdot - \tilde{z}(u))}{\left( u(\cdot + \tilde{z}(u)), \partial_{x}^{2} \phi_{c} \right)_{L^{2}}} \in H^{\infty}(\mathbb{R})$.
    \end{enumerate}
\end{lem}
\begin{proof}
    See Bona--Souganidis--Strauss~\cite[Theorem 4.1]{BSS1987}.
\end{proof}

Now we let $I \subset \mathbb{R}$ be an interval and $u(t) \in C(I, H^{\sigma/2}(\mathbb{R}))$ be a solution to \eqref{eq:Intro_gKdV_DoublePower;gkdv_inst}.
Then we see that $\partial_{t}u(t) \in C(I, H^{-(1 + \sigma/2)}(\mathbb{R}))$.
We additionally assume that $u(t) \in U_{\varepsilon}(\phi_{c})$ for all $t \in I$ and all $\varepsilon \in (0, \varepsilon_{1})$.
By Lemma 4.6 of \cite{GSS87}, we obtain
\begin{align}
    \frac{d}{dt}z(t) &= \left\langle \partial_{t}u(t), \tilde{z}'(u(t)) \right\rangle_{H^{-(1 + \sigma/ 2)}, H^{1 + \sigma/2}} \\
    &= \left\langle \partial_{x}E'(u(t)), \tilde{z}'(u(t)) \right\rangle_{H^{-(1 + \sigma/2)}, H^{1 + \sigma/2}} \\
    &= - \left\langle E'(u(t)), \partial_{x}\tilde{z}'(u(t)) \right\rangle_{H^{-\sigma/2}, H^{\sigma/2}}. \label{eq:Instability_Modulation_0010;gkdv_inst}
\end{align}

Here we introduce a virial type functional.
First we define
\begin{equation}
    \Phi(x) \coloneq \int_{-\infty}^{x} \Lambda \phi_{c}(y) \, dy
\end{equation}
for $x \in \mathbb{R}$.
Since $\Lambda \phi_{c} \in L^{1}(\mathbb{R})$, we see that $\Phi$ is bounded in $\mathbb{R}$.
Next, we introduce a function $\rho \in C_{c}^{\infty}(\mathbb{R})$ satisfying that $0 \leq \rho(x) \leq 1$ for all $x \in \mathbb{R}$, $\rho(x) = 1$ for $|x| \leq 1$, and $\rho(x) = 0$ for $|x| \geq 2$, and put $\rho_{A}(x) \coloneq \rho(x/A)$ for $A \geq 1$.
Now we define
\begin{equation}
    \Phi_{A}(x) \coloneq \Phi(x)\rho_{A}(x)
\end{equation}
for $x \in \mathbb{R}$ and $A \geq 1$.
We can see that $\Phi_{A} \in H^{\infty}(\mathbb{R})$.
Using this function, we define a virial type functional $J_{A}$ as
\begin{equation}
    J_{A}(v) \coloneq \int_{\mathbb{R}} v(x + \tilde{z}(v)) \Phi_{A}(x) \, dx.
\end{equation}

\begin{lem} \label{lem:Instability_Lyapunov_BoundednessofJ;gkdv_inst}
    There exists $C > 0$ such that
    \begin{equation}
        |J_{A}(v)| \leq C A^{1/2}
    \end{equation}
    holds for all $A \geq 1$ and $v \in U_{\varepsilon_{1}}(\phi_{c})$.
\end{lem}
\begin{proof}
    First, we remark that there exists some $M > 0$ such that $\| v \|_{H^{\sigma/2}} \leq M$ holds for all $v \in U_{\varepsilon_{1}}(\phi_{c})$.
    Then, a direct calculation yields that
    \begin{align}
        |J_{A}(v)| & = \left| \int_{-2A}^{2A} v(x + \tilde{z}(v)) \Phi(x) \rho\left( \frac{x}{A} \right) \, dx \right| \\
        &\leq C \int_{-2A}^{2A} |v(x + \tilde{z}(v))| \, dx \\
        &\leq C A^{1/2} \| v \|_{L^{2}} \\
        &\leq CM A^{1/2}.
    \end{align}
    Therefore, replacing $CM$ with $C$ concludes the proof.
\end{proof}

Here we put
\begin{equation}
    F_{A}(v) \coloneq \langle \partial_{x} v(\cdot + \tilde{z}(v)), \Phi_{A} \rangle_{H^{-(1 - \sigma/2)}, H^{1 - \sigma/2}} =  - \int_{\mathbb{R}} v(x) \partial_{x}\Phi_{A}(x - \tilde{z}(v)) \, dx \label{eq:Instability_Lyapunov_0001;gkdv_inst}
\end{equation} 
for $v \in U_{\varepsilon_{1}}(\phi_{c})$.
Since $\Phi_{A} \in H^{\infty}(\mathbb{R})$ and $u(t) \in C^{1}(I, H^{-(1 + \sigma/2)}(\mathbb{R}))$, we see it by \eqref{eq:Instability_Modulation_0010;gkdv_inst} that
\begin{align}
    \frac{d}{dt} J_{A}(u(t)) &= \langle \partial_{t}u(t, \cdot + z(t)), \Phi_{A} \rangle_{H^{-(1 + \sigma/2)}, H^{1 + \sigma/2}} + \frac{d}{dt}z(t) F_{A}(u(t)) \\
    &= \langle \partial_{x}E'(u(t, \cdot + z(t))), \Phi_{A} \rangle_{H^{-(1 + \sigma/2)}, H^{1 + \sigma/2}} - F_{A}(u(t)) \langle E'(u(t)), \partial_{x}\tilde{z}'(u(t)) \rangle_{H^{-\sigma/2}, H^{\sigma/2}} \\
    &= - \langle E'(u(t)), \partial_{x} \Phi_{A}(\cdot - z(t)) + F_{A}(u(t)) \partial_{x} \tilde{z}'(u(t)) \rangle_{H^{-\sigma/2}, H^{\sigma/2}}
\end{align}
holds for $t \in I$.
Putting $\theta_{A}(v) \coloneq \partial_{x} \Phi_{A}(\cdot - \tilde{z}(v)) + F_{A}(v) \partial_{x} \tilde{z}'(v)$, we simply write
\begin{equation}
    \frac{d}{dt} J_{A}(u(t)) = - \left\langle E'(u(t)), \theta_{A}(u(t)) \right\rangle_{H^{-\sigma/2}, H^{\sigma/2}}. \label{eq:Instability_Lyapunov_0010;gkdv_inst}
\end{equation}
Moreover, since $(v, \partial_{x}\tilde{z}'(v))_{L^{2}} = 1$ holds for $v \in U_{\varepsilon_{1}}(\phi_{c})$ by (iii) of Lemma \ref{lem:Instability_Modulation_Modulation;gkdv_inst}, we obtain
\begin{align}
    \left\langle M'(v), \theta_{A}(v) \right\rangle_{H^{-\sigma/2}, H^{\sigma/2}} &= (v, \theta_{A}(v))_{L^{2}} \\
    &= -F_{A}(v) + F_{A}(v) (v, \partial_{x} \tilde{z}'(v))_{L^{2}} = 0.
\end{align}
Using this, we can rewrite \eqref{eq:Instability_Lyapunov_0010;gkdv_inst} as
\begin{equation}
    \frac{d}{dt} J_{A}(u(t)) = - \langle S_{c}'(u(t)), \theta_{A}(u(t)) \rangle_{H^{-\sigma/2}, H^{\sigma/2}}.
\end{equation}
Now we put
\begin{equation}
    P_{A}(v) \coloneq \langle S_{c}'(v), \theta_{A}(v) \rangle_{H^{-\sigma/2}, H^{\sigma/2}}
\end{equation}
for $v \in U_{\varepsilon_{1}}(\phi_{c})$.
We will show the following proposition to obtain the instability of travelling wave solutions.

\begin{prop} \label{prop:Instability_Lyapunov_EstofPA;gkdv_inst}
    There exist $A_{3} \geq 1$, $\mu_{3} > 0$, and $\varepsilon_{3} \in (0, \varepsilon_{1})$ such that
    \begin{equation}
        S_{c}(\phi_{c}) \leq S_{c}(v) + \mu_{3} |P_{A_{3}}(v)|
    \end{equation}
    holds for all $v \in U_{\varepsilon_{3}}(\phi_{c})$.
\end{prop}
In the following, we proceed the proof of Proposition \ref{prop:Instability_Lyapunov_EstofPA;gkdv_inst}.

\begin{lem} \label{lem:Instability_Lyapunov_ConvofThetaA;gkdv_inst}
    It holds that $\theta_{A}(\phi_{c}) \rightarrow \Lambda \phi_{c}$ in $H^{\sigma/2}(\mathbb{R})$ as $A \rightarrow +\infty$.
\end{lem}
\begin{proof}
    It is sufficient to show that $\theta_{A}(\phi_{c}) \rightarrow \Lambda \phi_{c}$ in $H^{1}(\mathbb{R})$ as $A \rightarrow + \infty$.

    By the definition of $\tilde{z}$, we obtain
    \begin{equation}
        \theta_{A}(\phi_{c}) = \partial_{x} \Phi_{A} - \frac{F_{A}(\phi_{c})}{\| \partial_{x} \phi_{c} \|_{L^{2}}^{2}} \partial_{x}^{2} \phi_{c}. \label{eq:Instability_Lyapunov_0020;gkdv_inst}
    \end{equation}

    First, we show that 
    \begin{equation}
        F_{A}(\phi_{c}) \rightarrow 0 \label{eq:Instability_Lyapunov_0021;gkdv_inst}
    \end{equation}
    as $A \rightarrow +\infty$.
    A direct calculation yields that
    \begin{equation}
        F_{A}(\phi_{c}) = -\int_{\mathbb{R}} \phi_{c}(x) \Lambda \phi_{c}(x) \rho_{A}(x) \, dx - \frac{1}{A} \int_{\mathbb{R}} \phi_{c}(x) \Phi(x) \partial_{x} \rho\left( \frac{x}{A} \right) \, dx. \label{eq:Instability_Lyapunov_0030;gkdv_inst}
    \end{equation}
    We see that the first term of \eqref{eq:Instability_Lyapunov_0030;gkdv_inst} converges to zero as $A \rightarrow + \infty$ due to $(\phi_{c}, \Lambda \phi_{c})_{L^{2}} = 0$.
    The second term of \eqref{eq:Instability_Lyapunov_0030;gkdv_inst} is estimated as
    \begin{equation}
        \frac{1}{A} \left| \int_{\mathbb{R}} \phi_{c}(x) \Phi(x) \partial_{x} \rho\left( \frac{x}{A} \right) \, dx \right| \leq \frac{\| \partial_{x} \rho \|_{L^{\infty}} \| \phi_{c} \|_{L^{\infty}} \| \Phi \|_{L^{1}}}{A}. \\
    \end{equation}
    Then we obtain \eqref{eq:Instability_Lyapunov_0021;gkdv_inst}.

    Next, we show that $\theta_{A}(\phi_{c}) \rightarrow \Lambda \phi_{c}$ in $L^{2}(\mathbb{R})$ as $A \rightarrow +\infty$.
    By \eqref{eq:Instability_Lyapunov_0020;gkdv_inst}, we obtain
    \begin{equation}
        \| \theta_{A}(\phi_{c}) - \Lambda \phi_{c} \|_{L^{2}} \leq \| \partial_{x}\Phi_{A} - \Lambda \phi_{c} \|_{L^{2}} + \frac{|F_{A}(\phi_{c})|}{\| \partial_{x} \phi_{c} \|_{L^{2}}^{2}} \| \partial_{x}^{2} \phi_{c} \|_{L^{2}}. \label{eq:Instability_Lyapunov_0050;gkdv_inst}
    \end{equation}
    The first term of \eqref{eq:Instability_Lyapunov_0050;gkdv_inst} is estimated as
    \begin{align}
        \| \partial_{x} \Phi_{A} - \Lambda \phi_{c} \|_{L^{2}} &= \left\| \rho_{A} \Lambda \phi_{c} + \frac{1}{A} \Phi \partial_{x}\rho\left( \frac{\cdot}{A} \right) - \Lambda \phi_{c} \right\|_{L^{2}} \\
        &\leq \| (\rho_{A} - 1) \Lambda \phi_{c} \|_{L^{2}} + \frac{1}{A} \left\| \Phi \partial_{x}\rho\left( \frac{\cdot}{A} \right) \right\|_{L^{2}} \\
        &\leq \| (\rho_{A} - 1) \Lambda \phi_{c} \|_{L^{2}} + \frac{\| \Phi \|_{L^{\infty}} \| \partial_{x}\rho \|_{L^{2}}}{A^{1/2}} . \label{eq:Instability_Lyapunov_0060;gkdv_inst}
    \end{align}
    Therefore, we obtain it by \eqref{eq:Instability_Lyapunov_0050;gkdv_inst} and \eqref{eq:Instability_Lyapunov_0060;gkdv_inst} that $\| \theta_{A}(\phi_{c}) - \Lambda \phi_{c} \|_{L^{2}} \rightarrow 0$ as $A \rightarrow + \infty$.

    Finally, it is seen by direct calculations that
    \begin{align}
        \partial_{x} \theta_{A}(\phi_{c}) &= \partial_{x}^{2} \Phi_{A} - \frac{F_{A}(\phi_{c})}{\| \partial_{x} \phi_{c} \|_{L^{2}}} \partial_{x}^{3} \phi_{c}, \\
        \partial_{x}^{2} \Phi_{A}(x) &= \rho_{A}(x) \partial_{x} \Lambda \phi_{c}(x) + \frac{2}{A} \Lambda \phi_{c}(x) \partial_{x} \rho\left( \frac{x}{A} \right) + \frac{1}{A^{2}} \Phi(x) \partial_{x}^{2}\rho\left( \frac{x}{A} \right).
    \end{align}
    Then, similarly above, we can obtain that $\partial_{x} \theta_{A}(\phi_{c}) \rightarrow \partial_{x} \Lambda \phi_{c}$ in $L^{2}(\mathbb{R})$ as $A \rightarrow +\infty$.

    Hence, the proof is accomplished.
\end{proof}

\begin{lem} \label{lem:Instability_Lyapunov_IneqofPa_1;gkdv_inst}
    There exist $A_{2} \geq 1$, $\mu_{2} > 0$, and $\varepsilon_{2} > 0$ such that
    \begin{equation}
        S_{c}(v + \mu \theta_{A}(v)) \leq S_{c}(v) + \mu P_{A}(v)
    \end{equation}
    holds for all $A \in (A_{2}, \infty)$, $\mu \in (-\mu_{2}, \mu_{2})$, and $v \in U_{\varepsilon_{2}}(\phi_{c})$.
\end{lem}
\begin{proof}
    Let $\mu \in \mathbb{R}$ and $v \in U_{\varepsilon_{1}}(\phi_{c})$ be fixed.
    By the Taylor expansion, we obtain
    \begin{align}
        S_{c}(v + \mu \theta_{A}(v)) =& S_{c}(v) + \mu \langle S_{c}'(v), \theta_{A}(v) \rangle \\
        &+ \mu^{2} \int_{0}^{1} (1 - s) \langle S_{c}''(v + \mu s \theta_{A}(v)) \theta_{A}(v), \theta_{A}(v) \rangle \, ds. \label{eq:Instability_Lyapunov_0070;gkdv_inst}
    \end{align}
    
    For $\mu \in \mathbb{R}$ and $v \in U_{\varepsilon_{1}}(\phi_{c})$, we put
    \begin{equation}
        R(\mu, v) \coloneq \langle S_{c}''(v + \mu \theta_{A}(v)) \theta_{A}(v), \theta_{A}(v) \rangle.
    \end{equation}
    Then, we see it by Lemma \ref{lem:Instability_Modulation_Modulation;gkdv_inst} that
    \begin{equation}
        R(\mu, v(\cdot + y)) = R(\mu, v) \label{eq:Instability_Lyapunov_0071;gkdv_inst}
    \end{equation}
    holds for all $y \in \mathbb{R}$.
    Then, by the continuity of the mapping $\mathbb{R} \times H^{\sigma/2}(\mathbb{R}) \ni (\mu, v) \mapsto R(\mu, v)$, \eqref{eq:Instability_Lyapunov_0071;gkdv_inst}, \eqref{eq:Instability_UnstableDir_0030;gkdv_inst}, and Lemma \ref{lem:Instability_Lyapunov_ConvofThetaA;gkdv_inst}, we see that there exist $A_{2} \geq 1$, $\mu_{2} > 0$, and $\varepsilon_{2} \in (0, \varepsilon_{1})$ such that
    \begin{equation}
        \langle S_{c}''(v + \mu \theta_{A}(v)) \theta_{A}(v), \theta_{A}(v) \rangle \leq \frac{1}{2} \langle S_{c}''(\phi_{c}) \Lambda \phi_{c}, \Lambda \phi_{c} \rangle < 0 \label{eq:Instability_Lyapunov_0090;gkdv_inst}
    \end{equation}
    holds for $A \in (A_{2}, \infty)$, $\mu \in (-\mu_{2}, \mu_{2})$, and $v \in U_{\varepsilon_{2}}(\phi_{c})$.
    Combining this with \eqref{eq:Instability_Lyapunov_0070;gkdv_inst}, we conclude the statement.
\end{proof}

\begin{lem} \label{lem:Instability_Lyapunov_ConstructionofNehari0;gkdv_inst}
    There exist $A_{3} \geq 1$, $\mu_{3} > 0$, and $\varepsilon_{3} > 0$ such that, for any $v \in U_{\varepsilon_{3}}(\phi_{c})$, there exists $\mu_{\ast} \in (-\mu_{3}, \mu_{3})$ which satisfies that
    \begin{equation}
        K_{c}(v + \mu_{\ast} \theta_{A_{3}}(v)) = 0, \quad v + \mu_{\ast} \theta_{A_{3}}(v) \neq 0.
    \end{equation}
\end{lem}
\begin{proof}
    Since $\langle S_{c}''(\phi_{c}) \Lambda \phi_{c}, \Lambda \phi_{c} \rangle < 0$, we obtain it by Lemma \ref{lem:Variational_PositivenessofCoupling;gkdv_inst} that $\langle K_{c}'(\phi_{c}), \Lambda \phi_{c} \rangle \neq 0$.

    First we consider the case that $\langle K_{c}'(\phi_{c}), \Lambda \phi_{c} \rangle > 0$ holds.
    For $v \in U_{\varepsilon_{2}}(\phi_{c})$ and $A \geq A_{2}$, by the Taylor expansion with respect to $\mu \in \mathbb{R}$, we obtain
    \begin{equation}
        K_{c}(v + \mu \theta_{A}(v)) = K_{c}(v) + \mu \int_{0}^{1} \langle K_{c}'(v + \mu s \theta_{A}(v)), \theta_{A}(v) \rangle \, ds. \label{eq:Instability_Lyapunov_0095;gkdv_inst}
    \end{equation}
    Therefore, with similar argument for obtaining \eqref{eq:Instability_Lyapunov_0090;gkdv_inst} in the proof of Lemma \ref{lem:Instability_Lyapunov_IneqofPa_1;gkdv_inst}, we can see that there exist $A_{3} \in [A_{2}, \infty)$, $\mu_{3} \in (0, \mu_{2})$, and $\varepsilon_{3} \in (0, \varepsilon_{2}]$ such that
    \begin{equation}
        \langle K_{c}'(v + \mu \theta_{A_{3}}(v)), \theta_{A_{3}}(v) \rangle \geq \frac{1}{2} \langle K_{c}'(\phi_{c}), \Lambda \phi_{c} \rangle > 0 \label{eq:Instability_Lyapunov_0100;gkdv_inst}
    \end{equation}
    holds for $v \in U_{\varepsilon_{3}}(\phi_{c})$ and $\mu \in (-\mu_{3}, \mu_{3})$.
    By $K_{c}(\phi_{c}(\cdot - \tilde{z}(v))) = K_{c}(\phi_{c}) = 0$, we can retake $\varepsilon_{3} > 0$ so small that
    \begin{equation}
        |K_{c}(v)| \leq \frac{\mu_{3}}{2} \langle K_{c}'(\phi_{c}), \Lambda \phi_{c} \rangle \label{eq:Instability_Lyapunov_0110;gkdv_inst}
    \end{equation}
    holds for $v \in U_{\varepsilon_{3}}(\phi_{c})$.

    Now we let $v \in U_{\varepsilon_{3}}(\phi_{c})$.
    If $K_{c}(v) < 0$ holds, then we see it by \eqref{eq:Instability_Lyapunov_0095;gkdv_inst}, \eqref{eq:Instability_Lyapunov_0100;gkdv_inst}, and \eqref{eq:Instability_Lyapunov_0110;gkdv_inst} that
    \begin{align}
        K_{c}(v + \mu_{3} \theta_{A_{3}}(v)) &= K_{c}(v) + \mu_{3} \int_{0}^{1} \langle K_{c}'(v + \mu_{3} s \theta_{A_{3}}(v)), \theta_{A_{3}}(v) \rangle \, ds \\
        &> - \frac{\mu_{3}}{2} \langle K_{c}'(\phi_{c}), \Lambda \phi_{c} \rangle + \frac{\mu_{3}}{2} \langle K_{c}'(\phi_{c}), \Lambda \phi_{c} \rangle = 0.
    \end{align}
    Since $K_{c}(v) < 0$, the continuity of the function $\mathbb{R} \ni \mu \mapsto K(v + \mu \theta_{A_{3}}(v))$ gives some $\mu_{\ast} \in (0, \mu_{3})$ such that $K_{c}(v + \mu_{\ast} \theta_{A_{3}}(v)) = 0$.

    If $K_{c}(v) > 0$, we can obtain it similarly that
    \begin{align}
        K_{c}(v - \mu_{3}\theta_{A_{3}}(v)) &= K_{c}(v) - \mu_{3} \int_{0}^{1} \langle K_{c}'(v + \mu_{3} s \theta_{A_{3}}(v)), \theta_{A_{3}}(v) \rangle \, ds \\
        &< \frac{\mu_{3}}{2} \langle K_{c}'(\phi_{c}), \Lambda \phi_{c} \rangle - \frac{\mu_{3}}{2} \langle K_{c}'(\phi_{c}), \Lambda \phi_{c} \rangle = 0,
    \end{align}
    Therefore, we can take some $\mu_{\ast} \in (-\mu_{3}, 0)$ such that $K_{c}(v + \mu_{\ast} \theta_{A_{3}}(v)) = 0$.

    If $K_{c}(v) = 0$, it is clear that $\mu_{\ast} = 0$.

    Finally, we retake $\varepsilon_{3}$ so small as to obtain $v + \mu_{\ast} \theta_{A_{3}}(v) \neq 0$ if necessary.

    When $\langle K_{c}'(\phi_{c}), \Lambda \phi_{c} \rangle < 0$ holds, we can prove the statement similarly above.
\end{proof}

Now we give the proof of Proposition \ref{prop:Instability_Lyapunov_EstofPA;gkdv_inst}.
\begin{proof}
    Let $\phi_{c}$ be a ground state solution to \eqref{eq:Intro_dSPc;gkdv_inst} for $c > 0$.
    By Lemma \ref{lem:Instability_Lyapunov_ConstructionofNehari0;gkdv_inst}, for any $v \in U_{\varepsilon_{3}}(\phi_{c})$, there exists $\mu_{\ast} \in (-\mu_{3}, \mu_{3})$ such that
    \begin{equation}
        K_{c}(v + \mu_{\ast} \theta_{A_{3}}(v)) = 0, \quad v + \mu_{\ast} \theta_{A_{3}}(v) \neq 0.
    \end{equation}
    Then, by \eqref{eq:Intro_GroundState_Characterization;gkdv_inst} and Lemma \ref{lem:Instability_Lyapunov_IneqofPa_1;gkdv_inst}, we obtain
    \begin{equation}
        S_{c}(\phi_{c}) \leq S_{c}(v + \mu_{\ast} \theta_{A_{3}}(v)) \leq S_{c}(v) + \mu_{\ast} P_{A_{3}}(v) \leq S_{c}(v) + \mu_{3} |P_{A_{3}}(v)|.
    \end{equation}
    Therefore, the proof is accomplished.
\end{proof}

Now we give the proof of Proposition \ref{prop:Instability_Sufficient;gkdv_inst} by contradiction.
\begin{proof}
    Let $\phi_{c}$ be a ground state solution to \eqref{eq:Intro_dSPc;gkdv_inst} for $c > 0$.

    First, we consider the function $(0, \infty) \ni \lambda \mapsto S_{c}(\phi_{c}^{\lambda})$.
    Since $\partial_{\lambda} S_{c}(\phi_{c}^{\lambda}) |_{\lambda = 1} = 0$ and $\partial_{\lambda}^{2} S_{c}(\phi_{c}^{\lambda}) |_{\lambda = 1} < 0$, we can find some $\mu_{1} > 0$ such that 
    \begin{equation}
        S_{c}(\phi_{c}^{\lambda}) < S_{c}(\phi_{c}) \label{eq:Instability_Lyapunov_0131;gkdv_inst}
    \end{equation}
    holds for $\lambda \in (1 - \mu_{1}, 1 + \mu_{1})$.

    Now we suppose that the travelling wave solution $\phi_{c}(x - ct)$ is stable.
    Then we see that there exists $\mu_{0} \in (-\mu_{1}, 0) \cup (0, \mu_{1})$ such that $u^{1 - \mu_{0}}(t) \in U_{\varepsilon_{3}}(\phi_{c})$ for all $t \geq 0$, where $u^{1 - \mu_{0}}(t) \in C([0, \infty), H^{\sigma/2}(\mathbb{R}))$ is a time-global solution to \eqref{eq:Intro_gKdV_DoublePower;gkdv_inst} with $u^{1 - \mu_{0}}(0) = \phi_{c}^{1 - \mu_{0}}$.
    Then, by Lemma \ref{lem:Instability_Lyapunov_BoundednessofJ;gkdv_inst}, we see that
    \begin{equation}
        |J_{A_{3}}(u^{1 - \mu_{0}}(t))| \leq C A_{3}^{1/2} \label{eq:Instability_Lyapunov_0140;gkdv_inst}
    \end{equation}
    holds for all $t \geq 0$.

    Now we put
    \begin{equation}
        \delta_{\mu_{0}} \coloneq S_{c}(\phi_{c}) - S_{c}(\phi_{c}^{1 - \mu_{0}}) > 0.
    \end{equation}
    Then, by Proposition \ref{prop:Instability_Lyapunov_EstofPA;gkdv_inst} and the conservation laws, we obtain
    \begin{equation}
        \delta_{\mu_{0}} = S_{c}(\phi_{c}) - S_{c}(u^{1 - \mu_{0}}(t)) \leq \mu_{3} |P_{A_{3}}(u^{1 - \mu_{0}}(t))|,
    \end{equation}
    that is,
    \begin{equation}
        |P_{A_{3}}(u^{1 - \mu_{0}}(t))| \geq \frac{\delta_{\mu_{0}}}{\mu_{3}} \label{eq:Instability_Lyapunov_0150;gkdv_inst}
    \end{equation}
    holds for all $t \geq 0$.
    Since
    \begin{equation}
        \frac{d}{dt} J_{A_{3}}(u^{1 - \mu_{0}}(t)) = -P_{A_{3}}(u^{1 - \mu_{0}}(t)),
    \end{equation}
    we see it by \eqref{eq:Instability_Lyapunov_0150;gkdv_inst} that $|J_{A_{3}}(u^{1 - \mu_{0}}(t))| \rightarrow \infty$ as $t \rightarrow + \infty$, which contradicts to \eqref{eq:Instability_Lyapunov_0140;gkdv_inst}.

    Hence, the proof is accomplished.
\end{proof}

\section{Proof of Theorem \ref{thm:Intro_Main_Instability;gkdv_inst}} \label{section:Proof_of_Theorem;gkdv_inst}

In this section, we prove Theorem \ref{thm:Intro_Main_Instability;gkdv_inst} by observing conditions of $p, q \in \mathbb{N}$ which satisfy the assumption of Proposition \ref{prop:Instability_Sufficient;gkdv_inst}.

First, we recall the convergence properties of ground state solutions to \eqref{eq:Intro_dSPc;gkdv_inst}.
Here we consider the stationary problem with a single power nonlinearity
\begin{equation}
    D_{x}^{\sigma} \psi + \psi - \psi^{q} = 0, \quad x \in \mathbb{R}. \label{eq:Main_0010;gkdv_inst}
\end{equation}
We let $\psi_{1, q} \in H^{\sigma/2}(\mathbb{R})$ be the unique positive, even ground state solution to \eqref{eq:Main_0010;gkdv_inst}.
Moreover, for a function $\phi$, we consider the function $\breve{\phi}$ given by the scaling
\begin{equation}
    \phi(x) = c^{1/(q-1)} \breve{\phi}(c^{1/\sigma} x). \label{eq:Main_0030;gkdv_inst}
\end{equation}

\begin{lem} \label{lem:Main_ConvofGS;gkdv_inst}
    \begin{enumerate}[label=\textup{(\Roman*) \ }]
        \item Assume one of conditions \textup{(I)} or \textup{(II-1)} in Theorem \ref{thm:Intro_dSPc_previous;gkdv_inst}.
        Let $\phi_{c} \in H^{\sigma/2}(\mathbb{R})$ be a positive ground state solution to \eqref{eq:Intro_dSPc;gkdv_inst} for $c > 0$, and let $\breve{\phi}_{c}$ be a function given by the scaling \eqref{eq:Main_0030;gkdv_inst}.
        Then it holds that $\breve{\phi}_{c} \rightarrow \psi_{1, q}$ in $H^{\sigma/2}(\mathbb{R})$ as $c \rightarrow +\infty$.

        \item Assume condition \textup{(II-2)} in Theorem \ref{thm:Intro_dSPc_previous;gkdv_inst}.
        Let $\phi_{c} \in H^{\sigma/2}(\mathbb{R})$ be a negative ground state solution to \eqref{eq:Intro_dSPc;gkdv_inst} for $c > 0$, and $\breve{\phi}_{c}$ be a function given by the scaling \eqref{eq:Main_0030;gkdv_inst}.
        Moreover, put $\chi_{1, q} \coloneq - \psi_{1, q}$.
        Then it holds that $\breve{\phi}_{c} \rightarrow \chi_{1, q}$ in $H^{\sigma/2}(\mathbb{R})$ as $c \rightarrow + \infty$.
    \end{enumerate}
\end{lem}
\begin{proof}
    See \cite[Lemma 4.3]{kokubu2025stability}.
\end{proof}

Here we let $\phi_{c}$ be one of ground state solutions to \eqref{eq:Intro_dSPc;gkdv_inst} obtained in Theorem \ref{thm:Intro_dSPc_previous;gkdv_inst}, and $\phi_{c}^{\lambda}$ be a function given by \eqref{eq:Instability_0010;gkdv_inst} for $\lambda > 0$.
By a direct calculation, we can see that
\begin{equation}
    S_{c}(\phi_{c}^{\lambda}) = \frac{\lambda^{\sigma}}{2} \| D_{x}^{\sigma/2} \phi_{c} \|_{L^{2}}^{2} + \frac{c}{2} \| \phi_{c} \|_{L^{2}}^{2} - \frac{a \lambda^{(p-1)/2}}{p+1} \int_{\mathbb{R}} \phi_{c}^{p+1}\, dx - \frac{\lambda^{(q-1)/2}}{q+1} \int_{\mathbb{R}} \phi_{c}^{q+1} \, dx.
\end{equation}
Differentiating this with respect to $\lambda$, we obtain
\begin{align}
    0 =& \langle S_{c}'(\phi_{c}), \Lambda \phi_{c} \rangle = \partial_{\lambda} S_{c}(\phi_{c}^{\lambda}) |_{\lambda=1} \\
    &= \frac{\sigma}{2} \| D_{x}^{\sigma/2} \phi_{c} \|_{L^{2}}^{2} - \frac{a(p-1)}{2(p+1)} \int_{\mathbb{R}} \phi_{c}^{p+1} \, dx - \frac{q-1}{2(q+1)} \int_{\mathbb{R}} \phi_{c}^{q+1} \, dx, \label{eq:Main_0040;gkdv_inst}
\end{align}
and
\begin{align}
    & \langle S_{c}''(\phi_{c}) \Lambda \phi_{c}, \Lambda \phi_{c} \rangle = \partial_{\lambda}^{2} S_{c}(\phi_{c}^{\lambda}) |_{\lambda=1} \\
    &= \frac{\sigma (\sigma - 1)}{2} \| D_{x}^{\sigma/2} \phi_{c} \|_{L^{2}}^{2} - \frac{a(p-1)(p-3)}{4(p+1)} \int_{\mathbb{R}} \phi_{c}^{p+1} \, dx - \frac{(q-1)(q-3)}{4(q+1)} \int_{\mathbb{R}} \phi_{c}^{q+1} \, dx. \label{eq:Main_0050;gkdv_inst}
\end{align}
Combining \eqref{eq:Main_0040;gkdv_inst} and \eqref{eq:Main_0050;gkdv_inst} yields
\begin{equation}
    \partial_{\lambda}^{2} S_{c}(\phi_{c}^{\lambda}) |_{\lambda=1} = \frac{a(p-1)(2\sigma + 1 - p)}{4(p+1)} \int_{\mathbb{R}} \phi_{c}^{p+1} \, dx + \frac{(q-1)(2\sigma + 1 - q)}{4(q+1)} \int_{\mathbb{R}} \phi_{c}^{q+1} \, dx. \label{eq:Main_0060;gkdv_inst}
\end{equation}

First we assume condition (I) in Theorem \ref{thm:Intro_dSPc_previous;gkdv_inst}, where $a = +1$ and $q$ is odd.
\begin{lem} \label{lem:Main_SuffCondUnderI;gkdv_inst}
    Let $a = +1$, $q$ is odd, and $\phi_{c}$ be a positive ground state solution to \eqref{eq:Intro_dSPc;gkdv_inst} for $c > 0$.
    Then the followings hold:
    \begin{enumerate}[label=\textup{(\roman*)} \ ]
        \item If $2\sigma + 1 \leq p < q$, then $\partial_{\lambda}^{2} S_{c}(\phi_{c}^{\lambda}) |_{\lambda=1} < 0$ holds for all $c > 0$.
        \item If $p < 2\sigma + 1 < q$, then there exists $c_{1} > 0$ such that $\partial_{\lambda}^{2} S_{c}(\phi_{c}^{\lambda}) |_{\lambda=1} < 0$ holds for all $c \in (c_{1}, \infty)$.
    \end{enumerate}
\end{lem}

\begin{proof}
    Under this condition, \eqref{eq:Main_0060;gkdv_inst} can be written as
    \begin{equation}
        \partial_{\lambda}^{2} S_{c}(\phi_{c}^{\lambda}) |_{\lambda=1} = \frac{(p-1)(2\sigma + 1 - p)}{4(p+1)} \| \phi_{c} \|_{L^{p+1}}^{p+1} + \frac{(q-1)(2\sigma + 1 - q)}{4(q+1)} \| \phi_{c} \|_{L^{q+1}}^{q+1}. \label{eq:Main_0070;gkdv_inst}
    \end{equation}

    If $2\sigma + 1 \leq p < q$, it is clear that the right hand side of \eqref{eq:Main_0070;gkdv_inst} is negative for all $c > 0$.

    Now we assume that $p < 2\sigma + 1 < q$.
    Then a direct calculation yields that
    \begin{equation}
        \partial_{\lambda}^{2} S_{c}(\phi_{c}^{\lambda}) |_{\lambda=1} < 0 \Longleftrightarrow \frac{\| \phi_{c} \|_{L^{p+1}}^{p+1}}{\| \phi_{c} \|_{L^{q+1}}^{q+1}} < \frac{p+1}{q+1} \cdot \frac{(q-1)\{ q - (2\sigma + 1) \}}{(p-1) (2\sigma + 1 - p)}. \label{eq:Main_0080;gkdv_inst}
    \end{equation}
    Using the scaling \eqref{eq:Main_0030;gkdv_inst}, we see that
    \begin{equation}
        \frac{\| \phi_{c} \|_{L^{p+1}}^{p+1}}{\| \phi_{c} \|_{L^{q+1}}^{q+1}} = c^{-\beta} \frac{\| \breve{\phi}_{c} \|_{L^{p+1}}^{p+1}}{\| \breve{\phi}_{c} \|_{L^{q+1}}^{q+1}} \label{eq:Main_0090;gkdv_inst}
    \end{equation}
    with $\beta = (q-p)/(q-1) > 0$.
    Then, by (i) of Lemma \ref{lem:Main_ConvofGS;gkdv_inst}, we can see that there exists $c_{1} > 0$ such that \eqref{eq:Main_0080;gkdv_inst} holds for $c \in (c_{1}, \infty)$.
\end{proof}

Next we assume condition (II-1) in Theorem \ref{thm:Intro_dSPc_previous;gkdv_inst}, where $a = -1$ and $p$ is odd.
\begin{lem} \label{lem:Main_SuffCondUnderII1;gkdv_inst}
    Let $a = -1$, $p$ is odd, and $\phi_{c}$ be a positive ground state solution to \eqref{eq:Intro_dSPc;gkdv_inst} for $c > 0$.
    \begin{enumerate}[label=\textup{(\roman*)} \ ]
        \item If $p < q = 2\sigma + 1$ or $p \leq 2\sigma + 1 < q$, then $\partial_{\lambda}^{2} S_{c}(\phi_{c}^{\lambda}) |_{\lambda=1} < 0$ holds for all $c > 0$.
    \item If $2\sigma + 1 < p < q$, then there exists $c_{2} > 0$ such that $\partial_{\lambda}^{2} S_{c}(\phi_{c}^{\lambda}) |_{\lambda=1} < 0$ holds for all $c \in (c_{2}, \infty)$.
    \end{enumerate}
\end{lem}
\begin{proof}
    In this case, \eqref{eq:Main_0060;gkdv_inst} can be written as
    \begin{equation}
        \partial_{\lambda}^{2} S_{c}(\phi_{c}^{\lambda}) |_{\lambda=1} = - \frac{(p-1)(2\sigma + 1 - p)}{4(p+1)} \| \phi_{c} \|_{L^{p+1}}^{p+1} + \frac{(q-1)(2\sigma + 1 - q)}{4(q+1)} \| \phi_{c} \|_{L^{q+1}}^{q+1}. \label{eq:Main_0100;gkdv_inst}
    \end{equation}

    If $p < q = 2\sigma + 1$ or $p \leq 2\sigma + 1 < q$, it can be easily seen that the right hand side of \eqref{eq:Main_0100;gkdv_inst} is negative for all $c > 0$, which implies that statement (i) holds.

    If $2\sigma + 1 < p < q$, we can see that
    \begin{equation}
        \partial_{\lambda}^{2} S_{c}(\phi_{c}^{\lambda}) |_{\lambda=1} < 0 \Longleftrightarrow \frac{\| \phi_{c} \|_{L^{p+1}}^{p+1}}{\| \phi_{c} \|_{L^{q+1}}^{q+1}} < \frac{p+1}{q+1} \cdot \frac{(q-1)\{ q - (2\sigma + 1) \}}{(p-1) \{ p - (2\sigma + 1) \}}. \label{eq:Main_0110;gkdv_inst}
    \end{equation}
    Thanks to the scaling \eqref{eq:Main_0030;gkdv_inst} and Lemma \ref{lem:Main_ConvofGS;gkdv_inst}, we can show that there exists some $c_{2} > 0$ such that \eqref{eq:Main_0110;gkdv_inst} holds for all $c \in (c_{2}, \infty)$ with similar way to the proof of Lemma \ref{lem:Main_SuffCondUnderI;gkdv_inst}.
\end{proof}

Finally, we assume condition (II-2) in Theorem \ref{thm:Intro_dSPc_previous;gkdv_inst}, where $a = -1$, $p$ is even, and $q$ is odd.
In this case, we can obtain a negative ground state solution $\phi_{c}$ to \eqref{eq:Intro_dSPc;gkdv_inst}.
Noting that $\phi_{c}$ is negative, we can rewrite \eqref{eq:Main_0060;gkdv_inst} as
\begin{equation}
    \partial_{\lambda}^{2} S_{c}(\phi_{c}^{\lambda}) |_{\lambda=1} = \frac{(p-1)(2\sigma + 1 - p)}{4(p+1)} \| \phi_{c} \|_{L^{p+1}}^{p+1} + \frac{(q-1)(2\sigma + 1 - q)}{4(q+1)} \| \phi_{c} \|_{L^{q+1}}^{q+1},
\end{equation}
which coincides with \eqref{eq:Main_0070;gkdv_inst}.
Therefore, with almost the same way as Lemma \ref{lem:Main_SuffCondUnderI;gkdv_inst}, we can show the following statement.
\begin{lem} \label{lem:Main_SuffCondUnderII2;gkdv_inst}
    Let $a = -1$, $p$ is even, $q$ is odd, and $\phi_{c}$ be a negative ground state solution to \eqref{eq:Intro_dSPc;gkdv_inst} for $c > 0$.

    \begin{enumerate}[label=\textup{(\roman*)} \ ]
        \item If $2\sigma + 1 \leq p < q$, then $\partial_{\lambda}^{2} S_{c}(\phi_{c}^{\lambda}) |_{\lambda=1} < 0$ holds for all $c > 0$.
        \item If $p < 2\sigma + 1 < q$, then there exists $c_{3} > 0$ such that $\partial_{\lambda}^{2} S_{c}(\phi_{c}^{\lambda}) |_{\lambda=1} < 0$ holds for all $c \in (c_{3}, \infty)$.
    \end{enumerate}
\end{lem}

Finally, we conclude Theorem \ref{thm:Intro_Main_Instability;gkdv_inst} by Proposition \ref{prop:Instability_Sufficient;gkdv_inst}, and Lemmas \ref{lem:Main_SuffCondUnderI;gkdv_inst}, \ref{lem:Main_SuffCondUnderII1;gkdv_inst}, and \ref{lem:Main_SuffCondUnderII2;gkdv_inst}.

\section{Decay estimate of ground state solutions} \label{section:DecayEst_of_GS;gkdv_inst}

In this section, we prove that any ground state solution to \eqref{eq:Intro_dSPc;gkdv_inst} has a polynomial decay at infinity when $1 \leq \sigma < 2$.
Actually, we can obtain the desired estimate for any nontrivial solutions as follows:

\begin{thm} \label{prop:DecayEst;gkdv_inst}
    Let $1 \leq \sigma < 2$ and $\phi \in H^{\sigma/2}(\mathbb{R})$ be a nontrivial solution to \eqref{eq:Intro_dSPc;gkdv_inst}.
    Then it holds that, for any $l \in \mathbb{Z}_{+}$, there exists some $C > 0$ such that
    \begin{equation}
        |\partial_{x}^{l} \phi(x)| \leq C \langle x \rangle^{-(1 + \sigma) - l} \label{eq:DecayEst_0010;gkdv_inst}
    \end{equation}
    for all $x \in \mathbb{R}$.
\end{thm}

We prove Theorem \ref{prop:DecayEst;gkdv_inst} following the method in \cite{Riano-Roudenko_2022} with some modifications.
Hereafter, we always assume that $1 \leq \sigma < 2$.

Here we let $\phi_{c} \in H^{\sigma/2}(\mathbb{R})$ be a nontrivial solution to \eqref{eq:Intro_dSPc;gkdv_inst} and consider the function $\breve{\phi}$ obtained by the scaling \eqref{eq:Main_0030;gkdv_inst}.
Then we see that $\breve{\phi}_{c}$ solves the following equation:
\begin{equation}
    D_{x}^{\sigma} \breve{\phi} + \breve{\phi} - a c^{-\beta} \breve{\phi}^{p} - \breve{\phi}^{q} = 0,
\end{equation}
where $\beta = (q - p)/(q - 1) > 0$.
Therefore, it suffices to show that the estimate \eqref{eq:DecayEst_0010;gkdv_inst} holds for any nontrivial solutions $\phi \in H^{\sigma/2}(\mathbb{R})$ to the following equation:
\begin{equation}
    D_{x}^{\sigma} \phi + \phi - a_{1} \phi^{p} - a_{2} \phi^{q} = 0 \label{eq:DecayEst_ScaledEquation;gkdv_inst}
\end{equation}
with $a_{j} \in \mathbb{R}$ ($j = 1, 2$) are coefficients.

Now we introduce the function $G_{\sigma}$ defined as
\begin{equation}
    G_{\sigma}(x) \coloneq \frac{1}{2 \pi} \int_{\mathbb{R}} \frac{e^{i \xi x}}{1 + |\xi|^{\sigma}} \, dx = \frac{1}{\sqrt{2 \pi}} \mathscr{F}^{-1} \left[ \frac{1}{1 + |\xi|^{\sigma}} \right](x), \quad x \in \mathbb{R}.
\end{equation}
Moreover, we define the operator $1/(1 + D_{x}^{\sigma})$ as
\begin{equation}
    \frac{1}{1 + D_{x}^{\sigma}} f \coloneq G_{\sigma} \ast f,
\end{equation}
where $f$ is an appropriate function.
Then, a nontrivial solution $\phi$ to \eqref{eq:DecayEst_ScaledEquation;gkdv_inst} is written as
\begin{equation}
    \phi(x) = \left( G_{\sigma} \ast (a_{1} \phi^{p} + a_{2} \phi^{q}) \right)(x) = a_{1} \frac{1}{1 + D_{x}^{\sigma}} \phi^{p}(x) + a_{2} \frac{1}{1 + D_{x}^{\sigma}} \phi^{q}(x) \label{eq:DecayEst_00101;gkdv_inst}
\end{equation}
Here we recall some properties of $G_{\sigma}$.
\begin{lem} \label{lem:DecayEst_Properties_of_G;gkdv_inst}
    \begin{enumerate}[label=\textup{(\roman*)} \ ]
        \item $G_{\sigma} \in L^{1}(\mathbb{R})$.
        \item $G_{\sigma}$ is a positive and even function.
        \item There exists some $C_{\sigma} > 0$ such that $\lim_{|x| \rightarrow \infty} |x|^{1 + \sigma} G_{\sigma}(x) = C_{\sigma}$.
    \end{enumerate}
\end{lem}
The second and third properties have been observed by many authors (see e.g.\ \cite{Frank-Lenzmann}).
Moreover, the first property for the case that $1 < \sigma < 2$ is easily seen because the function $\mathbb{R} \ni \xi \mapsto 1/(1 + |\xi|^{\sigma})$ is bounded and continuous in $\mathbb{R}$, and belongs to $L^{1}(\mathbb{R})$.
However, the first property for $\sigma = 1$ is not immediate.
Therefore, we give the proof of it.
\begin{proof}[Proof of (i) for $\sigma = 1$]
    We shall observe the asymptotic behaviour of $G_{1}$ at infinity and at the origin.
    Due to the evenness of $G_{1}$, we may consider the function
    \begin{equation}
        g_{1}(x) \coloneq \int_{0}^{\infty} \frac{\cos(\xi x)}{1 + \xi} \, d\xi
    \end{equation}
    for $x > 0$.

    \underline{Step 1.} Asymptotic behaviour at infinity.

    In this step, we will show that $g_{1}(x) \leq C/x^{2}$ holds for $x \geq 1$.
    
    Integrating by parts, we obtain
    \begin{align}
        x g_{1}(x) &= \int_{0}^{\infty} \frac{\partial_{\xi}\{ \sin(\xi x) \}}{1 + \xi} \, d\xi = \int_{0}^{\infty} \frac{\sin(\xi x)}{(1 + \xi)^{2}} \, d\xi, \\
        x^{2} g_{1}(x) &= \int_{0}^{\infty} \frac{\partial_{\xi}\{1 - \cos(\xi x)\}}{(1 + \xi)^{2}} \, d\xi = 2 \int_{0}^{\infty} \frac{1 - \cos(\xi x)}{(1 + \xi)^{3}} \, d\xi. \label{eq:DecayEst_00102;gkdv_inst}
    \end{align}
    For $x \geq 1$.
    Moreover, for $x \geq 1$, we see that
    \begin{align}
        \left| \int_{0}^{\infty} \frac{\cos(\xi x)}{(1 + \xi)^{3}} \, dx \right| &= \left| \frac{1}{x} \int_{0}^{\infty} \frac{\partial_{\xi} \{ \sin(\xi x) \}}{(1 + \xi)^{3}} \, d\xi \right| \\
        &= \left| \frac{3}{x} \int_{0}^{\infty} \frac{\sin(\xi x)}{(1 + \xi)^{4}} \, d\xi \right| \\
        &\leq \frac{3}{x} \int_{0}^{\infty} \frac{d \xi}{(1 + \xi)^{4}} = \frac{C}{x}. \label{eq:DecayEst_00103;gkdv_inst}
    \end{align}
    Combining \eqref{eq:DecayEst_00102;gkdv_inst} with \eqref{eq:DecayEst_00103;gkdv_inst} gives that $g_{1}(x) \leq C/x^{2}$ holds for $x \geq 1$.

    \underline{Step 2.} Asymptotic behaviour at the origin.

    In this step, we will show that
    \begin{equation}
        \lim_{x \rightarrow +0} \frac{g_{1}(x)}{\log(1/x)} = 1.
    \end{equation}

    Let $x \in (0, 1)$.
    Then we see that
    \begin{align}
        g_{1}(x) &= \int_{0}^{\pi/x} \frac{d \xi}{1 + \xi} + \int_{0}^{\pi/x} \frac{\cos(\xi x) - 1}{1 + \xi} \, d\xi + \int_{\pi/x}^{\infty} \frac{\cos(\xi x)}{1 + \xi} \, d\xi \\
        &= \log \left( 1 + \frac{\pi}{x} \right) - \int_{0}^{\pi/x} \frac{1 - \cos(\xi x)}{1 + \xi} \, d\xi + \int_{\pi/x}^{\infty} \frac{\cos(\xi x)}{1 + \xi} \, d\xi. \label{eq:DecayEst_00104;gkdv_inst}
    \end{align}
    Since $0 \leq 1 - \cos \theta \leq \theta^{2}/2$ for $\theta \in \mathbb{R}$, we have
    \begin{align}
        0 \leq \int_{0}^{\pi/x} \frac{1 - \cos(\xi x)}{1 + \xi} \, d\xi &= \int_{0}^{\pi} \frac{1 - \cos \theta}{x + \theta} \, d\theta \\
        & \leq \int_{0}^{\pi} \frac{\theta^{2}}{2(x + \theta)} \, d\theta \leq \int_{0}^{\pi} \frac{\theta}{2} \, d\theta = \frac{\pi}{4}. \label{eq:DecayEst_00105;gkdv_inst}
    \end{align}
    Moreover, integrating by parts, we see that
    \begin{align}
        \left| \int_{\pi/x}^{\infty} \frac{\cos(\xi x)}{1 + \xi} \, d\xi \right| &= \left| \int_{\pi/x}^{\infty} \frac{\sin(\xi x)}{x(1 + \xi)^{2}} \, d\xi \right| \\
        &= \left| \int_{\pi}^{\infty} \frac{\sin \theta}{(x + \theta)^{2}} \, d\theta \right| \leq \int_{\pi}^{\infty} \frac{d\theta}{\theta^{2}} = \frac{1}{\pi}. \label{eq:DecayEst_00106;gkdv_inst}
    \end{align}
    Therefore, \eqref{eq:DecayEst_00104;gkdv_inst}, \eqref{eq:DecayEst_00105;gkdv_inst}, and \eqref{eq:DecayEst_00106;gkdv_inst} yields that
    \begin{equation}
        \left| g_{1}(x) - \log \frac{1}{x} \right| \leq \log(\pi + x) + \frac{\pi^{2}}{4} + \frac{1}{\pi} \leq \log (\pi + 1) + \frac{\pi^{2}}{4} + \frac{1}{\pi} = C
    \end{equation}
    for $x \in (0, 1)$.
    Namely, we see that
    \begin{equation}
        \left| \frac{g_{1}(x)}{\log(1/x)} - 1 \right| \leq \frac{C}{\log(1/x)} \rightarrow 0
    \end{equation}
    as $x \rightarrow +0$.

    Finally, these asymptotic behaviours give that $g_{1} \in L^{1}(\mathbb{R})$.
    
    Hence, the proof is accomplished.
\end{proof}

From the third property of Lemma \ref{lem:DecayEst_Properties_of_G;gkdv_inst}, it immediately follows that
\begin{equation}
    G_{\sigma}(x) \leq C |x|^{-(1 + \sigma)} \label{eq:DecayEst_0011;gkdv_inst}
\end{equation}
for $|x| > 1$ with some constant $C > 0$.

\begin{lem} \label{lem:DecayEst_H_infinity;gkdv_inst}
    Let $\phi \in H^{\sigma/2}(\mathbb{R})$ be a nontrivial solution to \eqref{eq:DecayEst_ScaledEquation;gkdv_inst}.
    Then it holds that $\phi \in H^{\infty}(\mathbb{R})$.
\end{lem}
\begin{proof}
    See \cite[Lemma 2.6]{Kokubu2024}.
\end{proof}

Now we observe the spatial decay of a nontrivial solution to \eqref{eq:DecayEst_ScaledEquation;gkdv_inst}.
First, we observe the decay estimate of a nontrivial solution itself.

\begin{prop} \label{prop:DecayEst_0th_derivative;gkdv_inst}
    Let $\phi \in H^{\sigma/2}(\mathbb{R})$ be a nontrivial solution to \eqref{eq:DecayEst_ScaledEquation;gkdv_inst}.
    Then there exists some $C > 0$ such that
    \begin{equation}
        |\phi(x)| \leq C \langle x \rangle^{-(1 + \sigma)} \label{eq:DecayEst_0020;gkdv_inst}
    \end{equation}
    for all $x \in \mathbb{R}$.
\end{prop}

Proposition \ref{prop:DecayEst_0th_derivative;gkdv_inst} can be shown with the similar method to Amick--Toland~\cite[pp.23--24]{Amick-Toland_1}.

Let $\phi \in H^{\sigma/2}(\mathbb{R})$ be a nontrivial solution to \eqref{eq:DecayEst_ScaledEquation;gkdv_inst} and $\delta > 0$ be arbitrary.
Noting that $\phi \in H^{\infty}(\mathbb{R})$, we can take $X_{\delta} > 0$ such that $|\phi(x)| < \delta$ holds for $|x| > X_{\delta}$.
For $\alpha \in \{ 0, 1 + \sigma \}$, we put
\begin{equation}
    C_{\alpha, \delta} \coloneq C\left( \{ x \in \mathbb{R} : |x| > X_{\delta} \} \right).
\end{equation}
For $v \in C_{\alpha, \delta}$, we define
\begin{equation}
    \| v \|_{\alpha, \delta} \coloneq \sup_{|x| > X_{\delta}} (1 + |x|^{\alpha}) |v(x)|,
\end{equation}
so that $( C_{\alpha, \delta}, \| \cdot \|_{\alpha, \delta} )$ is a Banach space.
Moreover, we easily see that $C_{1 + \sigma, \delta} \hookrightarrow C_{0, \delta}$.

Here we set $g(\phi) \coloneq a_{1} \phi^{p-1} + a_{2} \phi^{q-1}$ and put
\begin{align}
    A_{\delta}(x) &\coloneq \int_{|y| \leq X_{\delta}} G_{\sigma}(x-y) g(\phi(y))\phi(y) \, dy, \\
    T_{\delta}[v](x) &\coloneq \int_{|y| \geq X_{\delta}} G_{\sigma}(x-y) g(\phi(y)) v(y) \, dy
\end{align}
for $v \in C_{\alpha, \delta}$ and $x \in \mathbb{R}$.
Moreover, for $v \in C_{\alpha, \delta}$, we define
\begin{equation}
    \Phi_{\delta}[v](x) \coloneq A_{\delta}(x) + T_{\delta}[v](x).
\end{equation}
Then we can see that $\Phi_{\delta}$ is a contraction map on $C_{\alpha, \delta}$ when $\delta > 0$ is sufficiently small.
Therefore, by the contraction mapping theorem, we can obtain $\Phi_{\delta}[\phi] = \phi$ in $C_{1 + \sigma, \delta}$, which implies that $\phi \in C_{1 + \sigma, \delta}$.
For details, see Amick--Toland~\cite{Amick-Toland_1}.

Next, we observe the decay estimate of derivatives of a nontrivial solution to \eqref{eq:DecayEst_ScaledEquation;gkdv_inst}.
To complete the proof of Theorem \ref{prop:DecayEst;gkdv_inst}, it suffices to prove the following lemma.
\begin{lem} \label{lem:DecayEst_nth_derivative_induction;gkdv_inst}
    Let $\phi \in H^{\sigma/2}(\mathbb{R})$ be a nontrivial solution to \eqref{eq:DecayEst_ScaledEquation;gkdv_inst}.
    Then, the following statement \textup{$\text{(\textasteriskcentered)}_{\kappa}$} holds for all $\kappa \in \mathbb{Z}_{+}$.
    \begin{itemize}
        \item[\textup{$\text{(\textasteriskcentered)}_{\kappa}$}] For all $l \in \mathbb{Z}_{+}$, there exists some $C > 0$ such that
        \begin{equation}
            |\partial_{x}^{l} \phi(x)| \leq C \langle x \rangle^{-(1 + \sigma) - \min\{l, \kappa \}} \label{eq:DecayEst_0061;gkdv_inst}
        \end{equation}
        for all $x \in \mathbb{R}$.
    \end{itemize}
\end{lem}

We will prove Lemma \ref{lem:DecayEst_nth_derivative_induction;gkdv_inst} with an induction argument with respect to $\kappa \in \mathbb{Z}_{+}$.

The following lemma plays an important role for the proof of Lemma \ref{lem:DecayEst_nth_derivative_induction;gkdv_inst}.
\begin{lem} \label{lem:DecayEst_FLS_ConvolutionLimit;gkdv_inst}
    Let $K, f \in L^{1}(\mathbb{R})$ satisfy
    \begin{equation}
        |K(x)| \leq C |x|^{-\beta}, \quad |f(x)| \leq C \langle x \rangle^{-\beta}
    \end{equation}
    with some $\beta > 1$.
    Assume that
    \begin{equation}
        \lim_{|x| \rightarrow \infty} |x|^{\beta} K(x) = K_{0}, \quad \lim_{|x| \rightarrow \infty} |x|^{\beta}f(x) = 0
    \end{equation}
    with some $K_{0} \in \mathbb{R}$.
    Then it holds that
    \begin{equation}
        \lim_{|x| \rightarrow \infty} |x|^{\beta} (K \ast f)(x) = K_{0} \int_{\mathbb{R}} f(y) \, dy.
    \end{equation}
\end{lem}
\begin{proof}
    See \cite[Lemma C.3]{Frank-Lenzmann-Silvestre}.
\end{proof}

Now we proceed the proof of Lemma \ref{lem:DecayEst_nth_derivative_induction;gkdv_inst}.

First, we show that $\text{(\textasteriskcentered)}_{\kappa}$ holds for $\kappa = 0$.
Namely, we shall prove that
\begin{equation}
    |\partial_{x}^{l} \phi(x)| \leq C \langle x \rangle^{-(1 + \sigma)} \label{eq:DecayEst_0140;gkdv_inst}
\end{equation}
holds for all $l \in \mathbb{Z}_{+}$.

When $l = 0$, inequality \eqref{eq:DecayEst_0140;gkdv_inst} follows from Proposition \ref{prop:DecayEst_0th_derivative;gkdv_inst}.

Next, we consider the case that $l \geq 1$.
By \eqref{eq:DecayEst_00101;gkdv_inst}, we obtain
\begin{equation}
    \partial_{x}^{l} \phi = a_{1} \frac{\partial_{x}^{l}}{1 + D_{x}^{\sigma}} \phi^{p} + a_{2} \frac{\partial_{x}^{l}}{1 + D_{x}^{\sigma}} \phi^{q}. \label{eq:DecayEst_0150;gkdv_inst}
\end{equation}
Therefore, it suffices to consider decay estimate of terms
\begin{equation}
    \frac{\partial_{x}^{l}}{1 + D_{x}^{\sigma}} \phi^{r} = G_{\sigma} \ast \partial_{x}^{l} \{ \phi^{r} \}
\end{equation}
for $r \in \{ p, q \}$.

When $l = 1$, since $\phi \in H^{\infty}(\mathbb{R})$ and $r \geq 2$, we have
\begin{equation}
    | \partial_{x} \{ \phi^{r}(x) \} | = |r \phi^{r-1}(x) \partial_{x}\phi(x)| \leq C \| \partial_{x} \phi \|_{L^{\infty}} \langle x \rangle^{-(r - 1)(1 + \sigma)} \leq C \langle x \rangle^{-(1 + \sigma)}.
\end{equation}
Then, by Lemma \ref{lem:DecayEst_FLS_ConvolutionLimit;gkdv_inst}, we obtain
\begin{equation}
    |x|^{1 + \sigma} \left( G_{\sigma} \ast \phi^{r-1} \partial_{x}\phi \right)(x) \rightarrow C
\end{equation}
as $|x| \rightarrow +\infty$ with some constant $C \in \mathbb{R}$.
This implies that
\begin{equation}
    \left| \frac{\partial_{x}}{1 + D_{x}^{\sigma}} \phi^{r} (x) \right| \leq \frac{C}{|x|^{1 + \sigma}} \label{eq:DecayEst_0160;gkdv_inst}
\end{equation}
for $|x| \geq 1$.
Therefore, by \eqref{eq:DecayEst_0150;gkdv_inst}, \eqref{eq:DecayEst_0160;gkdv_inst}, and the continuity of $\partial_{x} \phi$, we conclude that
\begin{equation}
    |\partial_{x} \phi(x)| \leq C \langle x \rangle^{-(1 + \sigma)}.
\end{equation}

When $l \geq 2$, we can prove \eqref{eq:DecayEst_0140;gkdv_inst} by induction.
Assume that
\begin{equation}
    |\partial_{x}^{k} \phi(x)| \leq C \langle x \rangle^{-(1 + \sigma)}
\end{equation}
holds for all $k \in \mathbb{Z}$ with $0 \leq k \leq l$.
By the Leibniz rule, we obtain
\begin{equation}
    \partial_{x}^{l+1} \{ \phi^{r}(x) \} = \sum_{k=0}^{l} \partial_{x}^{k} \{ \phi^{r-1}(x) \} \partial_{x}^{l+1-k} \phi(x) + \phi^{r-1}(x) \partial_{x}^{l+1} \phi(x)
\end{equation}
By the inductive assumption, we can see that
\begin{equation}
    |\partial_{x}^{k} \{ \phi^{r-1}(x) \} \partial_{x}^{l+1-k} \phi(x)| \leq C \langle x \rangle^{-r(1 + \sigma)} \leq C \langle x \rangle^{-(1 + \sigma)}
\end{equation}
for each $0 \leq k \leq l$.
Moreover, since $\phi \in H^{\infty}(\mathbb{R})$, we have
\begin{equation}
    |\phi^{r-1}(x) \partial_{x}^{l+1}\phi(x)| \leq C \| \partial_{x}^{l+1} \phi \|_{L^{\infty}} \langle x \rangle^{-(r-1)(1+\sigma)} \leq C \langle x \rangle^{-(1 + \sigma)}.
\end{equation}
Therefore, we obtain
\begin{equation}
    |\partial_{x}^{l+1} \{\phi^{r}(x)\}| \leq C \langle x \rangle^{-(1 + \sigma)}.
\end{equation}
By Lemma \ref{lem:DecayEst_FLS_ConvolutionLimit;gkdv_inst} again, we can see that
\begin{equation}
    \left| \frac{\partial_{x}^{l + 1}}{1 + D_{x}^{\sigma}} \phi^{r}(x) \right| = | G_{\sigma} \ast \partial_{x}^{l + 1}\{\phi^{r}\}(x)| \leq \frac{C}{|x|^{1 + \sigma}}
\end{equation}
holds for $|x| \geq 1$.

Hence, statement $\text{(\textasteriskcentered)}_{\kappa}$ for $\kappa = 0$ is proved.

Next, we consider the proof that statement $\text{(\textasteriskcentered)}_{\kappa}$ holds for $\kappa \in \mathbb{Z}_{+}$.
To prove $\text{(\textasteriskcentered)}_{\kappa}$ by induction, we need some lemmas.

\begin{lem}\label{lem:DecayEst_boundedness_dxdf;gkdv_inst}
    Let $\kappa \in \mathbb{Z}_{+}$ and assume that statement \textup{$\text{(\textasteriskcentered)}_{\kappa}$} holds for $\kappa \in \mathbb{Z}_{+}$.
    Let $\gamma, \beta, \eta \in \mathbb{Z}_{+}$ satisfy that $\beta \leq \kappa + 1$ and
    \begin{equation}
        \beta - \gamma - \eta - (r-1)(1+\sigma) < 0.
    \end{equation}
    Then, there exists $H \in C(\mathbb{R})$ such that
    $H$ is bounded in $\mathbb{R}$, $|H(x)| \rightarrow 0$ as $|x| \rightarrow \infty$, and that
    \begin{equation}
        \left| \partial_{x}^{\gamma} \{ x^{\beta} \partial_{x}^{\eta} \phi^{r}(x) \} \right| \leq H(x) \langle x \rangle^{-(1 + \sigma)}
    \end{equation}
    for all $x \in \mathbb{R}$.
\end{lem}
\begin{proof}
    See \cite[Lemma 2.5]{Riano-Roudenko_2022} and its appendix.
\end{proof}

\begin{lem} \label{lem:DecayEst_Commutator;gkdv_inst}
    Let $\beta \in \mathbb{N}$ and $f$ be a function satisfying that $\partial_{x}^{\beta'} \{ x^{\beta'} f \} \in L^{2}(\mathbb{R})$ holds for $0 \leq \beta' \leq \beta$.
    Then, it holds in $L^{2}$-sense that
    \begin{align}
        \left[ x^{\beta}, \frac{\partial_{x}^{\beta}}{1 + D_{x}^{\sigma}} \right]f &= \sum_{j=1}^{\beta} \sum_{k=1}^{j} \sum_{m=1}^{k} c_{\beta, j, k, m} \frac{D_{x}^{\sigma m -2k} \partial_{x}^{2k}}{(1 + D_{x}^{\sigma})^{1 + k}} \partial_{x}^{\beta - j}\{ x^{\beta - j} f \} + \sum_{j=0}^{\beta} c_{\beta, j} \frac{1}{1 + D_{x}^{\sigma}} \partial_{x}^{\beta - j}\{ x^{\beta - j} f \},
    \end{align}
    where $[A, B] \coloneq AB - BA$ for operators $A$ and $B$, and $c_{\beta, j, k, m}, c_{\beta, j} \in \mathbb{R}$ are coefficients.
\end{lem}
\begin{proof}
    See \cite[Lemma 2.6]{Riano-Roudenko_2022} and its appendix.
\end{proof}

For $k, m \in \mathbb{N}$ with $1 \leq m \leq k$ and a function $f$, we put
\begin{equation}
    \frac{D_{x}^{\sigma m -2k} \partial_{x}^{2k}}{(1 + D_{x}^{\sigma})^{1 + k}} f \coloneq K_{k, m} \ast f
\end{equation}
with
\begin{equation}
    K_{k, m}(x) \coloneq \mathscr{F}^{-1} \left[ \frac{D_{x}^{\sigma m -2k} \partial_{x}^{2k}}{(1 + D_{x}^{\sigma})^{1 + k}} \right] = \frac{1}{\sqrt{2 \pi}} \int_{\mathbb{R}} \frac{|\xi|^{\sigma m - 2k} (i\xi)^{2k}}{(1 + |\xi|^{\sigma})^{1 + k}} e^{i \xi x} \, d\xi.
\end{equation}

\begin{lem} \label{lem:DecayEst_K_in_Commutator;gkdv_inst}
    For any $k, m \in \mathbb{N}$ satisfying $1 \leq m \leq k$, it holds that $K_{k, m} \in L^{1}(\mathbb{R})$ and
    \begin{equation}
        |K_{k, m}(x)| \leq C|x|^{-(1 + \sigma)}
    \end{equation}
    for all $|x| \geq 1$ with some constant $C > 0$.
\end{lem}
\begin{proof}
    Let $k, m \in \mathbb{N}$ satisfy $1 \leq m \leq k$.
    Since the function $K_{k,m}$ is even, we may consider the function
    \begin{equation}
        K(x) \coloneq \int_{0}^{\infty} \frac{\xi^{\sigma m}}{(1 + \xi^{\sigma})^{1 + k}} \cos(x \xi) \, d\xi
    \end{equation}
    for $x \geq 0$ instead of $K_{k, m}(x)$.
    We can see that $K(x)$ is bounded in $[0, \infty)$

    Next, we observe the decay estimate of $K(x)$ at infinity.
    By integrating by parts, we obtain
    \begin{align}
        x K(x) &= \int_{0}^{\infty} \frac{\xi^{\sigma m}}{(1 + \xi^{\sigma})^{1 + k}} \partial_{\xi} \left\{ \sin(x \xi) \right\} \, d\xi \\
        &= \int_{0}^{\infty} \frac{b_{1} \xi^{\sigma (m + 1) - 1} + b_{2} \xi^{\sigma m - 1}}{(1 + \xi^{\sigma})^{2 + k}} \sin(x \xi) \, d\xi, \\
    \end{align}
    and
    \begin{align}
        x^{2} K(x) &= \int_{0}^{\infty} \frac{b_{1} \xi^{\sigma (m + 1) - 1} + b_{2} \xi^{\sigma m - 1}}{(1 + \xi^{\sigma})^{2 + k}} \partial_{\xi} \left\{ - \cos(x \xi) \right\} \, d\xi \\
        &= b_{1} \int_{0}^{\infty} \frac{\xi^{\sigma(m + 2) - 2}}{(1 + \xi^{\sigma})^{3 + k}} \cos(x \xi) \, d\xi + b_{2} \int_{0}^{\infty} \frac{\xi^{\sigma (m+1) - 2}}{(1 + \xi^{\sigma})^{3 + k}} \cos(x \xi) \, d\xi \\
        &\qquad \qquad+ b_{3} \int_{0}^{\infty} \frac{\xi^{\sigma m - 2}}{(1 + \xi^{\sigma})^{3 + k}} \cos(x \xi) \, d\xi \\
        & \eqcolon b_{1} I_{1}(x) + b_{2} I_{2}(x) + b_{3}I_{3}(x),
    \end{align}
    where $b_{j} \in \mathbb{R}$ ($j=1,2,3$) are coefficients which may vary line by line.

    First, we show that $x^{\sigma - 1} I_{1}(x) \rightarrow 0$ as $x \rightarrow +\infty$.
    With similar calculation above, we have
    \begin{align}
        xI_{1}(x) &= \int_{0}^{\infty} \frac{\xi^{\sigma(m+2) - 2}}{(1 + \xi^{\sigma})^{3 + k}} \partial_{\xi} \left\{ \sin(x \xi) \right\} \, d\xi \\
        &= c_{1} \int_{0}^{\infty} \frac{ \xi^{\sigma(m+2) - 3}}{(1 + \xi^{\sigma})^{3 + k}} \sin(x \xi) \, d\xi + c_{2} \int_{0}^{\infty} \frac{\xi^{\sigma(m+3) - 3}}{(1 + \xi^{\sigma})^{4 + k}} \sin(x \xi) \, d\xi
    \end{align}
    with some coefficients $c_{j} \in \mathbb{R}$ ($j=1,2$).
    Noting that $\sigma (m - k - 1) - 3 < 0$, we see that
    \begin{align}
        & \left| c_{1} \int_{0}^{\infty} \frac{ \xi^{\sigma(m+2) - 3}}{(1 + \xi^{\sigma})^{3 + k}} \sin(x \xi) \, d\xi + c_{2} \int_{0}^{\infty} \frac{\xi^{\sigma(m+3) - 3}}{(1 + \xi^{\sigma})^{4 + k}} \sin(x \xi) \, d\xi \right| \\
        &\leq C \left\{ \int_{0}^{\infty} \frac{\xi^{\sigma(m+2) - 3}}{(1 + \xi^{\sigma})^{3 + l}} \, d\xi + \int_{0}^{\infty} \frac{\xi^{\sigma(m+3) - 3}}{(1 + \xi^{\sigma})^{4 + k}} \, d\xi \right\} = C.
    \end{align}
    This implies that $|x I_{1}(x)| \leq C$ for $x \geq 1$, that is, $x^{\sigma - 1} I_{1}(x) \rightarrow 0$ as $x \rightarrow + \infty$.
    We can also see that $x^{\sigma - 1} I_{2}(x) \rightarrow 0$ as $x \rightarrow +\infty$ similarly above.

    Now we consider the asymptotic behaviour of $I_{3}(x)$ at infinity.
    If $m \geq 2$, we can see that $x^{\sigma - 1} I_{3}(x) \rightarrow 0$ as $x \rightarrow +\infty$ with similar way to $I_{1}$.
    If $m = 1$, we can see that
    \begin{align}
        x^{\sigma - 1}I_{3}(x) &= x^{\sigma - 1} \int_{0}^{\infty} \frac{\xi^{\sigma - 2}}{(1 + \xi^{\sigma})^{3 + k}} \cos(x \xi) \, d\xi \\
        &= \int_{0}^{\infty} \frac{\eta^{\sigma - 2}}{(1 + \left( \frac{\eta}{x} \right)^{\sigma})^{3 + k}} \cos \eta \, d\eta \\
        &= \int_{0}^{\pi} \frac{\eta^{\sigma - 2}}{(1 + \left( \frac{\eta}{x} \right)^{\sigma})^{3 + k}} \cos \eta \, d\eta + \int_{\pi}^{\infty} \frac{\eta^{\sigma - 2}}{(1 + \left( \frac{\eta}{x} \right)^{\sigma})^{3 + k}} \cos \eta \, d\eta. \label{eq:DecayEst_0200;gkdv_inst}
    \end{align}

    On the first term of \eqref{eq:DecayEst_0200;gkdv_inst}, since $\sigma - 2 > -1$, we obtain it by the dominated convergence theorem that
    \begin{equation}
        \int_{0}^{\pi} \frac{\eta^{\sigma - 2}}{(1 + \left( \frac{\eta}{x} \right)^{\sigma})^{1 + k}} \cos \eta \, d\eta \rightarrow \int_{0}^{\pi} \eta^{\sigma - 2} \cos \eta \, d\eta \label{eq:DecayEst_0201;gkdv_inst}
    \end{equation}
    as $x \rightarrow +\infty$.

    Next, we consider the second term of \eqref{eq:DecayEst_0200;gkdv_inst}.
    By integration by parts, we obtain
    \begin{align}
        &\int_{\pi}^{\infty} \frac{\eta^{\sigma - 2}}{(1 + \left( \frac{\eta}{x} \right)^{\sigma})^{3 + k}} \cos \eta \, d\eta \\
        &= d_{1} \int_{\pi}^{\infty} \frac{\eta^{\sigma - 3}}{(1 + \left( \frac{\eta}{x} \right)^{\sigma})^{3 + k}} \sin \eta \, d\eta + \frac{d_{2}}{x^{\sigma}} \int_{\pi}^{\infty} \frac{\eta^{2\sigma - 3}}{(1 + \left( \frac{\eta}{x} \right)^{\sigma})^{4 + k}} \sin \eta \, d\eta
    \end{align}
    with coefficients $c_{j} \in \mathbb{R}$ ($j = 1, 2$).
    The dominated convergence theorem yields that
    \begin{equation}
        \int_{\pi}^{\infty} \frac{\eta^{\sigma - 3}}{(1 + \left( \frac{\eta}{x} \right)^{\sigma})^{3 + k}} \sin \eta \, d\eta \rightarrow \int_{\pi}^{\infty} \eta^{\sigma - 3} \sin \eta \, d\eta
    \end{equation}
    as $x \rightarrow + \infty$.
    Moreover, since
    \begin{equation}
        \frac{1}{x^{\sigma}} \frac{\eta^{2\sigma - 3}}{(1 + \left( \frac{\eta}{x} \right)^{\sigma})^{4 + k}} \sin \eta = \frac{\left( \frac{\eta}{x} \right)^{\sigma}}{(1 + \left( \frac{\eta}{x} \right)^{\sigma})^{4 + k}} \eta^{\sigma - 3} \sin \eta \rightarrow 0
    \end{equation}
    as $x \rightarrow + \infty$ for each $\eta \in (\pi, \infty)$, we see that
    \begin{equation}
        \frac{1}{x^{\sigma}} \int_{\pi}^{\infty} \frac{\eta^{2\sigma - 3}}{(1 + \left( \frac{\eta}{x} \right)^{\sigma})^{4 + k}} \sin \eta \, d\eta \rightarrow 0
    \end{equation}
    as $x \rightarrow + \infty$.
    Therefore, we conclude that
    \begin{equation}
        \int_{\pi}^{\infty} \frac{\eta^{\sigma - 2}}{(1 + \left( \frac{\eta}{x} \right)^{\sigma})^{3 + k}} \cos \eta \, d\eta \rightarrow d_{2} \int_{\pi}^{\infty} \eta^{\eta - 3} \sin \eta \, d\eta \label{eq:DecayEst_0202;gkdv_inst}
    \end{equation}
    as $x \rightarrow + \infty$.
    Combining \eqref{eq:DecayEst_0200;gkdv_inst}, \eqref{eq:DecayEst_0201;gkdv_inst}, and \eqref{eq:DecayEst_0202;gkdv_inst}, we obtain
    \begin{equation}
        x^{\sigma - 1} I_{3}(x) \rightarrow \int_{0}^{\pi} \eta^{\sigma - 2} \cos \eta \, d\eta + d_{2} \int_{\pi}^{\infty} \eta^{\sigma - 3} \sin \eta \, d\eta
    \end{equation}
    as $x \rightarrow +\infty$.

    Finally, we obtain that
    \begin{align}
        x^{1 + \sigma} K(x) &= b_{1}x^{\sigma - 1} I_{1}(x) + b_{2} x^{\sigma - 1} I_{2}(x) + b_{3} x^{\sigma - 1} I_{3}(x) \\
        & \rightarrow
        \left\{
            \begin{aligned}
                & 0, & & \text{if} \ m \geq 2, \\
                & b_{3} \left( \int_{0}^{\pi} \eta^{\sigma - 2} \cos \eta \, d\eta + d_{2} \int_{\pi}^{\infty} \eta^{\sigma - 3} \sin \eta \, d\eta \right), & & \text{if} \ m = 1
            \end{aligned}
        \right.
    \end{align}
    as $x \rightarrow +\infty$.
    This implies the statement.
\end{proof}

Now we proceed the proof of statement $\text{(\textasteriskcentered)}_{\kappa}$ for $\kappa \geq 1$ by induction.

\begin{proof}
    We let $\kappa \in \mathbb{Z}_{+}$ be fixed and assume that
    \begin{equation}
        |\partial_{x}^{l} \phi(x)| \leq C \langle x \rangle^{-(1 + \sigma) - \min\{ l, \kappa \}}
    \end{equation}
    holds for all $l \in \mathbb{Z}_{+}$.
    Then we shall show that
    \begin{equation}
        |\partial_{x}^{l} \phi(x)| \leq C \langle x \rangle^{-(1 + \sigma) - (\kappa + 1)}
    \end{equation}
    holds for $l \geq \kappa + 1$.
    
    Let $l \in \mathbb{N}$ satisfy $l \geq \kappa + 1$, and put $\alpha = l - (\kappa + 1)$.
    Now we consider the term
    \begin{equation}
        x^{\kappa + 1} \frac{\partial_{x}^{l}}{1 + D_{x}^{\sigma}} \phi^{r} = x^{\kappa + 1} \frac{\partial_{x}^{\kappa + 1}}{1 + D_{x}^{\sigma}} \partial_{x}^{\alpha}\{ \phi^{r} \}
    \end{equation}
    with $r \in \mathbb{N}$ satisfying $r \geq 2$.
    Since the indices $\gamma = \kappa + 1 - j$, $\beta = \kappa + 1 - j$, and $\eta = \alpha$ satisfy the assumption of Lemma \ref{lem:DecayEst_boundedness_dxdf;gkdv_inst} for $0 \leq j \leq \kappa + 1$, we obtain
    \begin{equation}
        \left| \partial_{x}^{\kappa + 1 - j} \left\{ x^{\kappa + 1 - j} \partial_{x}^{\alpha} \{ \phi^{r} \} \right\}(x) \right| \leq C \langle x \rangle^{-(1 + \sigma)}, \label{eq:DecayEst_0210;gkdv_inst}
    \end{equation}
    which implies that $\partial_{x}^{\kappa + 1 - j} \left\{ x^{\kappa + 1 - j} \partial_{x}^{\alpha} \{ \phi^{r} \} \right\} \in L^{2}(\mathbb{R})$ for all $0 \leq j \leq \kappa + 1$.
    Then, by Lemma \ref{lem:DecayEst_Commutator;gkdv_inst}, we obtain
    \begin{align}
        x^{\kappa + 1} \frac{\partial_{x}^{\kappa + 1}}{(1 + D_{x}^{\sigma})} \partial_{x}^{\alpha} \{ \phi^{r} \} &= \left[ x^{\kappa + 1}, \frac{\partial_{x}^{\kappa + 1}}{1 + D_{x}^{\sigma}} \right] \partial_{x}^{\alpha} \{ \phi^{r} \} + \frac{\partial_{x}^{\kappa + 1}}{1 + D_{x}^{\sigma}} x^{\kappa + 1} \partial_{x}^{\alpha} \{ \phi^{r} \} \\
        &= \sum_{j=1}^{\kappa + 1} \sum_{k=1}^{j} \sum_{m=1}^{k} c_{\kappa, j, k, m} \frac{D_{x}^{\sigma m - 2k} \partial_{x}^{2k}}{(1 + D_{x}^{\sigma})^{1 + k}} \partial_{x}^{\kappa + 1 - j}\left\{ x^{\kappa + 1 - j} \partial_{x}^{\alpha} \{ \phi^{r} \} \right\} \\
        & \qquad \qquad + \sum_{j= 0}^{\kappa + 1} c_{\kappa, j} \frac{1}{1 + D_{x}^{\sigma}} \partial_{x}^{\kappa + 1 - j} \left\{ x^{\kappa + 1 - j} \partial_{x}^{\alpha} \{ \phi^{r} \} \right\}. \label{eq:DecayEst_0220;gkdv_inst}
    \end{align}
    
    First, we consider the first term of \eqref{eq:DecayEst_0220;gkdv_inst}.
    By \eqref{eq:DecayEst_0210;gkdv_inst} and Lemmas \ref{lem:DecayEst_FLS_ConvolutionLimit;gkdv_inst}, \ref{lem:DecayEst_K_in_Commutator;gkdv_inst}, we see it for each $j, k, m$ that
    \begin{equation}
        \left| \frac{D_{x}^{\sigma m - 2k} \partial_{x}^{2k}}{(1 + D_{x}^{\sigma})^{1 + k}} \partial_{x}^{\kappa + 1 - j} \left\{ x^{\kappa + 1 - j} \partial_{x}^{\alpha} \{ \phi^{r} \} \right\} (x) \right| = \left| K_{k, m} \ast \left( x^{\kappa + 1 - j} \partial_{x}^{\alpha} \{ \phi^{r} \} \right) \right| \leq C | x |^{-(1 + \sigma)} \label{eq:DecayEst_0230;gkdv_inst}
    \end{equation}
    holds for $|x| \geq 1$.
    
    Here we consider the second summation of \eqref{eq:DecayEst_0220;gkdv_inst}.
    Since \eqref{eq:DecayEst_0210;gkdv_inst} holds, we have it by Lemma \ref{lem:DecayEst_FLS_ConvolutionLimit;gkdv_inst} that
    \begin{equation}
        \left| \frac{1}{1 + D_{x}^{\sigma}} \partial_{x}^{\kappa + 1 - j} \left\{ x^{\kappa + 1 - j} \partial_{x}^{\alpha} \{ \phi^{r} \} \right\}(x) \right| = \left| G_{\sigma} \ast \left( \partial_{x}^{\kappa + 1 - j} \left\{ x^{\kappa + 1 - j} \partial_{x}^{\alpha} \{ \phi^{r} \} \right\} \right)(x) \right| \leq C|x|^{-(1 + \sigma)} \label{eq:DecayEst_0240;gkdv_inst}
    \end{equation}
    holds for $|x| \geq 1$.

    Therefore, combining \eqref{eq:DecayEst_0220;gkdv_inst}, \eqref{eq:DecayEst_0230;gkdv_inst}, and \eqref{eq:DecayEst_0240;gkdv_inst}, we obtain
    \begin{equation}
        \left| x^{\kappa + 1} \frac{\partial_{x}^{\kappa + 1}}{(1 + D_{x}^{\sigma})} \partial_{x}^{\alpha} \{ \phi^{r} \} (x) \right| \leq C|x|^{-(1 + \sigma)}
    \end{equation}
    for $|x| \geq 1$.

    Finally, since
    \begin{equation}
        x^{\kappa + 1} \partial_{x}^{l} \phi (x) = a_{1} x^{\kappa + 1} \frac{\partial_{x}^{l}}{1 + D_{x}^{\sigma}} \phi^{p} + a_{2} x^{\kappa + 1} \frac{\partial_{x}^{l}}{1 + D_{x}^{\sigma}} \phi^{q},
    \end{equation}
    we conclude that
    \begin{equation}
        |x^{\kappa + 1} \partial_{x}^{l} \phi(x)|\leq C |x|^{-(1 + \sigma)},
    \end{equation}
    that is,
    \begin{equation}
        |\partial_{x}^{l} \phi(x)| \leq C |x|^{-(1 + \sigma) - (\kappa + 1)}
    \end{equation}
    holds for $|x| \geq 1$.
    This complete the proof.
\end{proof}

\section*{Acknowledgments}
The author would like to express his gratitude to Professor Masahito Ohta for his encouragements and valuable comments on this study.

\bibliography{reference_list}
\bibliographystyle{amsplain}

\end{document}